\documentclass[11pt]{article}
\usepackage{amsmath}

\usepackage{color}
\usepackage{xcolor}

\definecolor{darkblue}{rgb}{0.0, 0.0, 0.8}
\definecolor{darkred}{rgb}{0.8, 0.0, 0.0}
\definecolor{darkgreen}{rgb}{0.6, 0.15, 0.15}

\usepackage{bbm}
\usepackage{url}
\usepackage[normalem]{ulem}
\usepackage{caption}
\usepackage{hyperref}
\hypersetup{
    colorlinks=true,
    urlcolor = {darkred},
    linkcolor = {darkred},
    citecolor = {darkred},
    linkcolor = {darkred},
    linktoc=all
}
\usepackage{amssymb}
\usepackage{fourier}
\usepackage{framed}
\usepackage{amsthm}
\usepackage{graphicx}
\usepackage{tikz}
\usepackage{tikz-cd}
\usepackage[inline]{enumitem}
\usepackage{mathtools}
\usepackage{etoolbox}
\usepackage{extpfeil}
\usepackage{soul}
\usepackage{authblk}
\usepackage[margin=0.8in]{geometry}
\usepackage{mathtools}

\usetikzlibrary{matrix,arrows,decorations.pathmorphing,decorations.pathreplacing,patterns,shapes.geometric}

\theoremstyle{definition}
\newtheorem{theorem}{Theorem}[section]

\newtheorem{proposition}[theorem]{Proposition}
\newtheorem{corollary}[theorem]{Corollary}

\newtheorem{remark}[theorem]{Remark}
\newtheorem{notation}[theorem]{Notation}
\newtheorem{definition}[theorem]{Definition}
\newtheorem{lemma}[theorem]{Lemma}

\newtheorem{problem}[theorem]{Problem}
\newtheorem{setup}[]{Setup}
\newtheorem{example}[theorem]{Example}

\newcommand{\bE}{\mathbf{E}}
\newcommand{\bbI}{\mathbb{I}}
\newcommand{\Ccal}{\mathcal{C}}
\newcommand{\Zplus}{\mathbf{Z}_{\geq 0}}

\newcommand{\sub}{\mathrm{sub}}

\newcommand{\intnbd}{\mathrm{nbd}_{\mathbb{I}}}
\newcommand{\intrk}{\mathrm{rk}_{\mathbb{I}}}
\newcommand{\intdgm}{\mathrm{dgm}_{\mathbb{I}}}

\newcommand{\bv}{\mathbf{v}}
\newcommand{\bw}{\mathbf{w}}

\newcommand{\ZZ}{\mathbf{ZZ}}

\newcommand{\overmax}{\max_{\ZZ} (I)}
\newcommand{\overmin}{\min_{\ZZ} (I)}

\newcommand{\nbd}{\mathrm{nbd}}
\newcommand{\rank}{\mathrm{rank}}
\newcommand{\Int}{\mathbf{Int}}
\newcommand{\A}{\mathcal{A}}

\newcommand{\I}{\mathbb{I}}
\newcommand{\id}{\mathrm{id}}
\newcommand{\Pb}{P}
\newcommand{\Q}{Q}

\newcommand{\vect}{\mathbf{vec}}

\newcommand{\F}{\mathbb{F}}
\newcommand{\Fcal}{\mathcal{F}}

\newcommand{\NN}{t}

\newcommand{\Z}{\mathbf{Z}}

\newcommand{\abs}[1]{\left\lvert{#1}\right\rvert}

\newcommand{\rk}{\mathrm{rk}}

\newcommand{\barc}{\mathrm{barc}}
\newcommand{\C}{\mathcal{C}}

\newcommand{\RNum}[1]{\uppercase\expandafter{\romannumeral #1\relax}}

\usepackage[ruled,vlined, linesnumbered]{algorithm2e}

\title{Computing Generalized Rank Invariant for 2-Parameter Persistence Modules via Zigzag Persistence and its Applications}

\author[1]{Tamal K. Dey}
\author[2]{Woojin Kim}
\author[3]{Facundo M\'emoli}

\affil[1]{Department of Computer Science, Purdue University\thanks{\texttt{tamaldey@purdue.edu}}}
\affil[2]{Department of Mathematics, Duke University\thanks{\texttt{woojin@math.duke.edu}}}
\affil[3]{Department of Mathematics and Department of Computer Science and Engineering, The Ohio State University\thanks{\texttt{memoli@math.osu.edu}}}

\begin{document}

\maketitle

\begin{abstract}
The notion of generalized rank invariant in the context of multiparameter persistence 
has become an important ingredient for defining interesting
homological structures such as generalized persistence
diagrams. Naturally, computing these rank invariants efficiently is a
prelude to computing any of these derived structures efficiently.
We show that the generalized rank over a finite interval $I$ of a $\Z^2$-indexed persistence
module $M$ is equal to the generalized rank of the zigzag module that is 
induced on a certain path in $I$ tracing mostly its
boundary. Hence, we can compute the generalized
rank over $I$ by computing the barcode of the zigzag module obtained
by restricting the bifiltration inducing $M$ to that path. 
If the bifiltration and $I$ have at most $\NN$ simplices and points respectively,
this computation takes $O(\NN^\omega)$ time where $\omega\in[2,2.373)$ is the exponent
of matrix multiplication.
Among others, we apply this result to obtain an
improved algorithm for the following problem. Given a bifiltration
inducing a module $M$,
determine whether $M$ is interval decomposable and, if so,
compute all intervals supporting its summands.
\end{abstract}

\section{Introduction}\label{sec:introdction}

In Topological Data Analysis (TDA) one of the central tasks is that of decomposing persistence modules into direct sums of \emph{indecomposables}. In the case of a persistence module $M$ over the integers $\Z$, the indecomposables are interval modules, which implies that $M$ is isomorphic to a direct sum of \emph{interval} modules $\mathbb{I}([b_\alpha,d_\alpha])$, for integers   $b_\alpha\leq d_\alpha$  and $\alpha$ in some index set $A$. This   follows from a classification theorem for quiver representations established by Pierre Gabriel in the 1970s. The multiset of intervals $\{[b_\alpha,d_\alpha],\,\alpha\in A\}$ that appear in this decomposition constitutes the \emph{persistence diagram}, or equivalently, the \emph{barcode} of $M$ --- a central object in TDA~\cite{DW22,EH2010}. 

There are many situations in which data naturally induce persistence modules over posets which are different from $\Z$  \cite{bauer2020cotorsion,carlsson2010zigzag,carlsson2010multiparameter,carlsson2009theory,DH21,escolar2016persistence,kim2021spatiotemporal,lesnick2012multidimensional,lesnick2015theory,miller2019modules}. Unfortunately, the situation already becomes `wild' 
when the domain poset is $\Z^2$.
In that situation, one must contend with the fact that a direct analogue of the notion of persistence diagrams may not exist \cite{carlsson2009theory}, namely it may not be possible to obtain a lossless up-to-isomorphism representation of the module as a direct sum of interval modules. 

Much energy has been put into finding ways in which one can extract incomplete but still stable  invariants from  persistence modules $M:\Z^d\rightarrow \vect$ (which we will refer to as a $\Z^d$-module).    Biasotti et al. \cite{biasotti2008multidimensional} proposed considering the restriction of a $\Z^d$-module to lines with positive slope. This was further developed by Lesnick and Wright in the RIVET project \cite{lesnick2015interactive} which facilitates the interactive visualization of $\Z^2$-modules. Cai et al. \cite{cai2020elder}  considered a certain elder-rule on the $\Z^2$-modules which arise in multiparameter clustering. Other  efforts have    identified algebraic conditions which can guarantee that $M$  can be decomposed into interval modules of varying degrees of complexity (e.g.  rectangle modules etc) \cite{botnan2020rectangle,cochoy2020decomposition}.
\begin{figure}[htbp]
    \centering
    \includegraphics[width=1\textwidth]{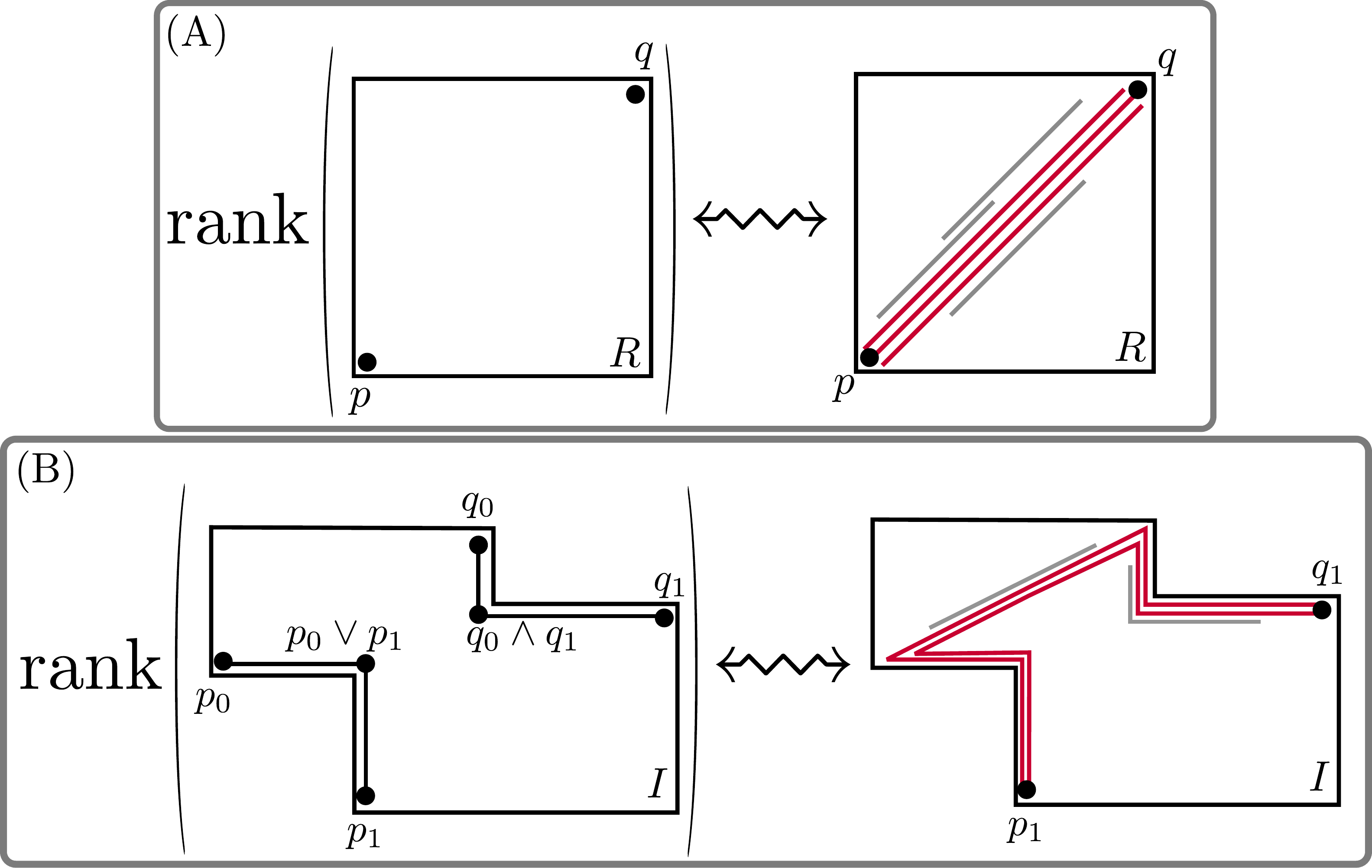}
     \caption{\textbf{Generalized rank via zigzag persistence.} Let $M$ be a $\Z^2$-module. (A) Standard rank: Let $p\leq q$ in $\Z^2$. The rank of the structure map from $p$ to $q$ coincides with the multiplicity of full bars (red) over the diagonal path, which is three. (B) Generalized rank: Let $I$ be an interval of $\Z^2$. Let us consider the zigzag poset $\partial I:p_1\leq (p_0\vee p_1) \geq p_0\leq  q_0 \geq (q_0\wedge q_1) \leq q_1$. The generalized rank of $M$ over $I$ is equal to the multiplicity of full bars (red) in the zigzag module $M_{\partial I}$, which is two (Theorem \ref{thm:rank is the multiplicity of full intervals}) (note: by definition, the zigzag poset $\partial I$ does not fully inherit the partial order on $\Z^2$. For example, the partial order on $\partial I$ does not contain the pair $(p_1,q_1)$ whereas $p_1\leq q_1$ in $\Z^2$).}
    \label{fig:intro}
\end{figure}

A distinct thread has been proposed by  Patel in \cite{patel2018generalized} through the reinterpretation of the persistence diagram of a $\Z$-module as the  M\"obius inversion of its rank function. Patel's work was then extended by Kim and M\'emoli \cite{kim2018generalized} to the setting of modules defined over any suitable locally finite poset. They generalized  the rank invariant via the \emph{limit-to-colimit map} over subposets and then conveniently expressed its M\"obius inversion. In fact the limit-to-colimit map was suggested by Amit Patel to the authors of \cite{kim2018generalized} who in \cite{kim2018generalized_v1} used it to define a notion of rank invariant for zigzag modules.  Chambers and Letscher \cite{chambers2018persistent} also  considered a notion of persistent homology over directed acyclic graphs using the limit-to-colimit map. Asashiba et al. \cite{asashiba2019approximation} study  the case of modules defined on an $m\times n$ grid and propose a high-level algorithm for computing both their generalized rank function and their M\"obius inversions with the goal of providing an approximation of a given module by interval decomposables. Asashiba et al.  \cite{asashiba2018interval} tackle the interval decomposability of a given $\Z^d$-module via quiver representation theory.

One fundamental algorithmic problem is that of determining whether a given $\Z^2$-module is interval decomposable, and if so, computing the intervals. 
There are some existing solutions to this problem in the literature.
Suppose that the input $\Z^2$-module is induced by a bifiltration comprising at most $t$ simplices on a grid of cardinality $O(t)$. 
First, the decomposition algorithm by Dey and Xin~\cite{dey2019generalized} can produce all
indecomposables from such a module in $O(t^{2\omega+1})$ time (see \cite{kerber2020multi} for comments about its implementation) where $\omega\in [2,2.373)$ is the exponent of matrix multiplication. Given these indecomposables, one could then test whether they are indeed interval modules. However, the algorithm requires that the input module be such that no two generators or relations
in the module have the same grade. Then, Asashiba et al.~\cite{asashiba2018interval}
give an algorithm which requires
enumerating an exponential number (in $t$) of intervals. Finally,
the algorithm by Meataxe sidesteps both of the above issues, but incurs a worst-case
cost of $O(t^{18})$ as explained in~\cite{dey2019generalized}.

\medskip
See also \cite{betthauser2021graded,botnan2021signed,bubenik2020virtual,kim2021bigraded,mccleary2020edit} for related recent work.

\paragraph{Contributions.}

One of our key results is the following. We prove that for an interval $I$ in $\Z^2$ we can compute the generalized rank invariant $\rk(M)(I)$ of a $\Z^2$-module $M$ through the computation of the zigzag persistence barcode of the \emph{restriction} of $M$ to the \emph{boundary cap} of $I$, which is a certain zigzag path in $I$; see Figure \ref{fig:intro} for an illustration.

\medskip These are our main results assuming that the input is a bifiltration with $O(t)$ simplices:

\begin{enumerate}
    \item We reduce the problem of computing the generalized rank invariant of a $\Z^2$-module to computing zigzag persistence (Theorem \ref{thm:rank is the multiplicity of full intervals}).

    \item We provide an algorithm {\sc Interval} (page \pageref{algorithm:interval}) to compute the barcode of any finite interval decomposable $\Z^2$-module in time $O(t^{\omega+2})$ (Proposition \ref{prop:timeInterval}). 
    
    \item We provide an algorithm {\sc IsIntervalDecomp} (page \pageref{algorithm:isintervaldecomp}) to decide the interval decomposability of a finite $\Z^2$-module in time $O(t^{3\omega+2})$ (Proposition \ref{prop:complexity for testing interval decomposability}). 
\end{enumerate}

\paragraph{Acknowledgements.} The authors thank the anonymous reviewers for constructive feedback and suggesting ideas that shortened the proof of Theorem \ref{thm:rank is the multiplicity of full intervals}. This work is supported by NSF grants CCF-2049010, CCF-1740761, DMS-1547357, and IIS-1901360.

\section{Preliminaries}\label{sec:preliminaries}

In \S\ref{sec:persistence modules}, we review the notion of interval decomposability of persistence modules. In \S\ref{sec:generalized rank invariant}, we review the notions of generalized rank invariant and generalized persistence diagram. In \S\ref{sec:canonical limits and colimits}, we discuss how to compute the limit and the colimit of a given functor $\Pb\rightarrow \vect$.

\subsection{Persistence Modules and their decompositions}\label{sec:persistence modules}

We fix a certain field $\F$ and every vector space in this paper is over $\F$. Let $\vect$ denote the category of \emph{finite dimensional} vector spaces  and linear maps over $\F$.

Let $\Pb$ be a poset. We regard $\Pb$ as the category that has points of $\Pb$ as objects. Also, for any $p,q\in \Pb$, there exists a unique morphism $p\rightarrow q$ if and only if $p\leq q$. For a positive integer $d$, let $\Z^d$ be given the partial order defined by $(a_1,a_2,\ldots,a_d)\leq (b_1,b_2,\ldots,b_d)$ if and only if $a_i\leq b_i$ for $i=1,2,\ldots,d$.  

A ($\Pb$-indexed) \textbf{persistence module} is any functor $M:\Pb\rightarrow \vect$ (which we will simply refer to as a $P$-module). In other words, to each $p\in \Pb$, a vector space $M_p$ is associated, and to each pair $p\leq q$ in $\Pb$, a linear map $\varphi_{M}(p,q):M_p\rightarrow M_q$ is associated. Importantly, whenever $p\leq q\leq r$ in $\Pb$, it must be that $\varphi_M(p,r)=\varphi_M(q,r)\circ \varphi_M(p,q)$. 

We say that a pair of $p,q\in \Pb$ is \textbf{comparable} if either $p\leq q$ or $q\leq p$. 

\begin{definition}[{\cite{botnan2018algebraic}}]\label{def:intervals}
An \textbf{interval} $I$ of $\Pb$ is a subset $I\subseteq \Pb$ such that: 	
\begin{enumerate*}[label=(\roman*)]
    \item $I$ is nonempty.    
    \item If $p,q\in I$ and $p\leq r\leq q$, then $r\in I$. \label{item:convexity}
    \item $I$ is \textbf{connected}, i.e. for any $p,q\in I$, there is a sequence $p=p_0,
		p_1,\cdots,p_\ell=q$ of elements of $I$ with $p_i$ and $p_{i+1}$ comparable for $0\leq i\leq \ell-1$.\label{item:interval3}
\end{enumerate*}
By $\Int(\Pb)$ we denote the set of all \emph{finite} intervals of $\Pb$. When $\Pb$ is finite and  connected, $P\in \Int(\Pb)$ will be referred to as  \textbf{the full interval}.
\end{definition}

For an interval $I$ of $\Pb$, the \textbf{interval module} $\I_I:\Pb\rightarrow \vect$ is defined as
\[\I_I(p)=\begin{cases}
\mathbb{F}&\mbox{if}\ p\in I,\\0
&\mbox{otherwise,} 
\end{cases}\hspace{20mm} \varphi_{\I_I}(p,q)=\begin{cases} \mathrm{id}_\mathbb{F}& \mbox{if} \,\,p,q\in I,\ p\leq q,\\ 0&\mbox{otherwise.}\end{cases}\]

Direct sums and quotients of $\Pb$-modules are defined pointwisely at each index $p\in \Pb$.

\begin{definition}\label{def:submodules and summands}Let $M$ be any $\Pb$-module. A \textbf{submodule} $N$ of $M$ is defined by subspaces $N_p\subseteq M_p$ such that $\varphi_M(p,q)(N_p)\subseteq N_q$ for all $p,q\in \Pb$ with $p\leq q$. These conditions guarantee that $N$ itself is a $\Pb$-module, with the structure maps given by the restrictions $\varphi_M(p,q)|_{N_p}$. In this case we write  $N\leq M$.

A submodule $N$ is a \textbf{summand} of $M$ if there exists a submodule $N'$ which is complementary to $N$, i.e. $M_p=N_p\oplus N'_p$ for all $p$. In that case, we say that $M$ is a \emph{direct sum} of $N,N'$ and write $M\cong N\oplus N'$.  Note that this direct sum is an \emph{internal} direct sum. 
\end{definition}

\begin{definition}\label{def:interval decomposable} A $\Pb$-module $M$ is called \emph{interval decomposable} if $M$ is isomorphic to a direct sum of interval modules, i.e. there exists an indexing set $\mathcal{J}$ such that $M\cong \bigoplus_{j\in \mathcal{J}}\I_{I_{j}}$ (external direct sum). In this case, the multiset $\{I_j: j\in \mathcal{J}\}$ is  called the \emph{barcode} of $M$, which will be denoted by $\barc(M)$.
\end{definition}
The Azumaya-Krull-Remak-Schmidt theorem guarantees that $\barc(M)$ is well-defined \cite{azumaya1950corrections}.
Consider a \textit{zigzag poset} of $n$ points, $    \bullet_{1}\leftrightarrow \bullet_{2} \leftrightarrow \ldots \bullet_{n-1} \leftrightarrow \bullet_n$
where $\leftrightarrow$ stands for either $\leq$ or $\geq$. A functor from a zigzag poset to $\vect$ is called a \textbf{zigzag module}  \cite{carlsson2010zigzag}. Any zigzag module is interval decomposable \cite{gabrielsthm} and thus admits a barcode.

The following proposition directly follows from the Azumaya-Krull-Remak-Schmidt theorem and will be useful  
in \S\ref{sec:computing interval summands and detecting interval decomposability}. 

\begin{proposition}\label{prop:quotienting does not change interver decomposabilty}
Let $M:\Pb\rightarrow \vect$ be interval decomposable and let $N\leq M$ is a summand of $M$ (Definition \ref{def:submodules and summands}). Then, $M/N$ is interval decomposable. 
\end{proposition}

\begin{proof} It suffices to show that $M\cong N\oplus (M/N)$ (since $M$ is interval decomposable, this isomorphism and Azumaya-Krull-Remak-Schmidt's theorem \cite{azumaya1950corrections} guarantee that $M/N$ is also interval decomposable).
Let $\iota:N\hookrightarrow M$ be the canonical inclusion. Since $N$ is a \emph{summand} of $M$, there exists a complementary submodule $N'\leq M$ with $M=N\oplus N'$  (internal direct sum). Hence, there exists $\pi:M\twoheadrightarrow N$ such that $\pi\circ \iota=\id_N$. Hence, by the splitting lemma \cite[Theorem 23.1]{munkres2018elements}, the exact sequence $0\rightarrow N \stackrel{\iota}{\rightarrow} M \rightarrow M/N \rightarrow 0$ splits, i.e. $M\cong N\oplus (M/N)$ (external direct sum), as desired.
\end{proof}

\subsection{Generalized rank invariant and generalized persistence diagrams }\label{sec:generalized rank invariant}

Let $\Pb$ be a finite connected poset and consider any $\Pb$-module $M$. Then $M$ admits a limit  $\varprojlim M=(L, (\pi_p:L\rightarrow M_p)_{p\in \Pb})$ and a colimit $\varinjlim M =(C, (i_p:M_p\rightarrow C)_{p\in \Pb})$; see Appendix \ref{sec:limits and colimits}. This implies that, for every $p\leq q$ in $\Pb$, $\varphi_M(p\leq q)\circ \pi_p =\pi_q\ \ \mbox{and }\  i_q\circ \varphi_M(p\leq q) =i_p$, which in turn imply $i_p \circ \pi_p=i_q \circ \pi_q:L\rightarrow C$ for any $p,q\in \Pb$.

\begin{definition}[{\cite{kim2018generalized}}]\label{def:rank}
The \textbf{canonical limit-to-colimit map $\psi_M:\varprojlim M\rightarrow \varinjlim M$} is the linear map $i_p\circ \pi_p$ for \emph{any} $p\in \Pb$. The \textbf{generalized rank} of $M$ is $\rank(M):=\rank(\psi_M)$. 
\end{definition}

The rank of $M$ counts the multiplicity of the fully supported interval modules in a direct sum decomposition of $M$.
\begin{theorem}[{\cite[Lemma 3.1]{chambers2018persistent}}]
\label{thm:rk}
The rank of $M$ is equal to the number of indecomposable summands of $M$ which are isomorphic to the interval module $\mathbb{I}_{\Pb}$. 
\end{theorem}

\begin{definition}\label{def:rank function} 
The \textbf{$(\Int)$-generalized rank invariant} of $M$ is the map $\intrk(M):\Int(\Pb)\rightarrow \Z_+$ defined as $I\mapsto \rank(M|_I)$, where $M|_I$ is the restriction of $M$ to $I$.
\end{definition}

\begin{definition}\label{def:IGPD} 
The \textbf{$(\Int)$-generalized persistence diagram of $M$} is the unique\footnote{The existence and uniqueness is guaranteed by properties of the M\"obius inversion formula \cite{rota1964foundations,stanley2011enumerative}.} function $\intdgm(M):\Int(\Pb)\rightarrow \Z$ that satisfies, for any $I\in \Int(\Pb)$, \[\intrk(M)(I)=\sum_{\substack{J\supseteq I\\ J\in\Int(\Pb)}}\intdgm(M)(J).\]
\end{definition}

The following is a slight variation of \cite[Theorem 3.14]{kim2018generalized} and \cite[Theorem 5.10]{asashiba2019approximation}.
\begin{theorem}\label{thm:dgm generalizes barc} If a given $M:\Pb\rightarrow\vect$ is interval decomposable, then for all $I\in \Int(\Pb)$,
$\intdgm(M)(I)$ is equal to the multiplicity of $I$ in $\barc(M)$.
\end{theorem}

\begin{proof} For $I\in \Int(\Pb)$, let $\mu_I\in\Zplus$ be the multiplicity of $I$ in $\barc(M)$. By invoking Theorem \ref{thm:rk}, we have that $\intrk(M)(I)=\sum_{\substack{J\supseteq I\\ J\in\Int(\Pb)}} \mu_J$ \cite[Proposition 3.17]{kim2018generalized}. By the uniqueness of $\intdgm(M)$ mentioned in Definition \ref{def:IGPD}, $\intdgm(M)(I)=\mu_I$ for all $I\in \Int(\Pb)$.
\end{proof}

We consider $P$ to be a 2d-grid and focus on the setting of $\Z^2$-modules.

\begin{definition}\label{def:interval neighborhood}For any $I\in\Int(\Z^2)$, we define $\intnbd(I):=\{p\in \Z^2\setminus I: I\cup\{p\}\in \Int(\Z^2)\}.$
\end{definition}
Note that $\intnbd(I)$ is nonempty \cite[Proposition  3.2]{asashiba2019approximation}. 
When $A\subseteq \intnbd(I)$ contains more than one point, $A\cup I$ is not necessarily an interval of $\Z^2$. However, there always  exists a unique smallest interval that contains $A\cup I$ which is denoted by $\overline{A\cup I}$.

\begin{remark}[{\cite[Theorem 5.3]{asashiba2019approximation}}] If in Definition \ref{def:IGPD} we assume that $\Pb\in \Int(\Z^2)$ 
then we have that for every $I\in\Int(\Pb)$,\footnote{In \cite{asashiba2019approximation}, only the case  $\Pb=\{1,\ldots,m\}\times \{1,\ldots,n\}\subset \Z^2$ was considered. However, it is not difficult to check that Eq.~(\ref{eq:formula}) is still valid for any finite interval $P$ in $\Z^2$ and any subinterval $I\subseteq P$.} 
\begin{equation}\label{eq:formula}
 \intdgm(M)(I)=\intrk(M)(I)+\sum_{\substack{A\subseteq \intnbd(I)\cap P\\ A\neq \emptyset }}(-1)^{\abs{A}}\intrk(M)\bigg(\overline{A\cup I}\bigg).
\end{equation}
\end{remark}

\subsection{Canonical constructions of limits and colimits}\label{sec:canonical limits and colimits}

Let $M$ be any $P$-module.
\begin{notation}\label{not:sim}
Let $p,q\in \Pb$ and let $v_{p}\in M_p$ and $v_{q}\in M_q$. We write $v_p\sim v_q$ if $p$ and $q$ are comparable, and either $v_p$ is mapped to $v_q$ via $\varphi_M(p,q)$ or $v_q$ is mapped to $v_p$ via $\varphi_M(q,p)$.
\end{notation}

The following proposition gives a standard way of constructing a limit and a colimit of a $P$-module $M$. Since it is well-known, we do not prove it (see for example \cite[Section E]{kim2018generalized}).

\begin{proposition}\label{prop:computation} 
	\begin{enumerate}[label=(\roman*)]
		\item \label{item:limit}The limit of $M$ is (isomorphic to) the pair $\left(W,(\pi_p)_{p\in \Pb}\right)$ where: 	\begin{equation}\label{eq:limit}
		    W:=\left\{(v_p)_{p\in \Pb}\in \bigoplus_{p\in \Pb} M_p:\ \forall p\leq q \mbox{ in } \Pb,\ v_{p}\sim v_{q} \right\}
		\end{equation}
		and for each $p\in \Pb$, the map $\pi_p:W\rightarrow M_p$ is the canonical projection. An element of $W$ is called a \textbf{section} of $M$.
		\item The colimit of $M$ is (isomorphic to) the pair $\left(U, (i_p)_{p\in \Pb}\right)$ described as follows: For $p\in \Pb$, let the map $j_p:M_p\hookrightarrow \bigoplus_{p\in \Pb}M_p$ be the canonical injection. $U$ is the quotient  $\left(\bigoplus_{p\in \Pb}M_p\right)/T$, where $T$ is the subspace of $\bigoplus_{p\in \Pb} M_p$ which is generated by the vectors of the form  $j_p(v_p)-j_q(v_q), \ v_p\sim v_{q},$ the map $i_p:M_p\rightarrow U$ is the composition $\rho\circ j_p$, where $\rho$ is the quotient map   $\bigoplus_{p\in \Pb} M_p\rightarrow U$.
		\label{item:colim}
	\end{enumerate}
\end{proposition}

\begin{framed}
\begin{setup}\label{convention}
In the rest of the paper, limits and colimits of a $\Pb$-module $M$ will all be constructed as in Proposition \ref{prop:computation}. Hence, assuming that $\Pb$ is connected, the canonical limit-to-colimit map $\varprojlim M \rightarrow \varinjlim M$ is $\psi_{M}:= i_p\circ \pi_p$ for any $p\in \Pb$. 

\end{setup} 
\end{framed}

\section{Computing generalized rank via boundary zigzags}\label{sec:generalized rank via boundary zigzag}
In \S \ref{sec:lower and upper fences} we introduce the notions of lower and upper fences of a poset. In \S \ref{sec:boundary cap}, we introduce the \emph{boundary cap} $\partial I$ of a finite interval $I$ of $\Z^2$, which is a path, a certain sequence of points in $I$. In \S\ref{sec: generalized rank via boundary}, we show that the rank of any functor $M:I\rightarrow \vect$ can be obtained by computing the barcode of the zigzag module over the path $\partial I$.

\subsection{Lower and upper fences of a poset}\label{sec:lower and upper fences}

Let $P$ be any connected poset. Given any $p\in P$, by $p^{\downarrow}$, we denote the set of all elements of $P$ that are less than or equal to $p$. Dually $p^{\uparrow}$ is defined as the set of all elements of $P$ that are greater than or equal to $p$. 

\begin{definition}\label{def:fences} A subposet $L\subset P$ (resp. $U\subset P$) is called a \emph{lower (resp. upper) fence} of $P$ if $L$ is connected, and for any $q\in P$, the intersection $L\cap q^{\downarrow}$ (resp. $U\cap q^{\uparrow}$) is nonempty and connected.
\end{definition}

\begin{proposition}\label{prop:section extension}\label{prop:commute}
Let $L$ and $U$ be a lower and an upper fences of $P$ respectively. Given any $P$-module $M$, we have
$\varprojlim M \cong \varprojlim M|_{L}$ and $\varinjlim M \cong \varinjlim M|_{U}$.
\end{proposition}

There are multiple proofs of Proposition \ref{prop:section extension}. One is given in an earlier version of \cite{kim2018generalized} (see Proposition D.14 in the second arXiv version). A concrete and tangible proof, which utilizes the following lemma, is given below. 

\begin{lemma}\label{lem:minimal-extent} Let $M$ be a $P$-module and let $L\subset P$ be a lower fence. Let $\bv\in \bigoplus_{p\in I}{M_p}$.
The tuple $\bv$ belongs to $\varprojlim M$ if and only if for  \emph{every} $p\in P$, it holds that $\bv_p=\varphi_M(q,p)(\bv_q)$ for
\emph{every} $q\in L$ such that $q\leq p$. 
\end{lemma}

\begin{proof}
The forward direction is straightforward from the description of $\varprojlim M$ given in Proposition \ref{prop:computation} \ref{item:limit}. Let us show the backward direction. Again, by Proposition \ref{prop:computation} \ref{item:limit} it suffices to show that for all $r\leq p$ in $P$ it holds that $\bv_p=\varphi_M(r,p)(\bv_r)$. Fix $r\leq p$ in $P$. Since $L$ is a lower fence, there exists $q\in  L$ such that $q\leq r\leq p$. Then, by assumption and by functoriality of $M$, we have:
\[\bv_p=\varphi_M(q,p)(\bv_q)=\big(\varphi_M(r,p)\circ \varphi_M(q,r)\big)(\bv_q)=\varphi_M(r,p)\big( \varphi_M(q,r)(\bv_q)\big)=\varphi_M(r,p)(\bv_r),\]
as desired.
\end{proof}

\begin{proof}[Proof of Proposition \ref{prop:section extension}]
We only prove $\varprojlim M \cong \varprojlim M|_L$. It suffices to prove that the section extension map $e:\varprojlim M|_L \rightarrow \varprojlim M$ in (\ref{eq:section extension}) is bijective. The injectivity is clear by definition. The surjectivity follows from the forward direction of the statement in Lemma \ref{lem:minimal-extent}. 
\end{proof}

The canonical isomorphism $\varprojlim M \cong \varprojlim M|_{L}$ in Proposition \ref{prop:section extension} is given by the canonical \emph{section extension} $e:\varprojlim M|_{L} \rightarrow \varprojlim M$. Namely,
\begin{equation}\label{eq:section extension}
    e: (\bv_{p})_{p\in L}\mapsto (\bw_{q})_{q\in P },
\end{equation}
where for any $q\in P$, the vector $\bw_q$ is defined as $\varphi_M(p,q)(\bv_p)$ for \emph{any} $p\in L\cap q^{\downarrow}$; the connectedness of $L\cap q^{\downarrow}$ guarantees that $\bw_q$ is well-defined. Also, if $q\in L$, then $\bw_q=\bv_q$. The inverse $r:=e^{-1}$ is the canonical section restriction. The other isomorphism $\varinjlim M \cong \varinjlim M|_{U}$ in Proposition \ref{prop:section extension} is given by the map $i:\varinjlim M|_U\rightarrow \varinjlim M$ defined by $[v_p]\mapsto [v_p]$ for any $p\in U$ and any $v_p\in M_p$; the fact that this map $i$ is well-defined will become clear from Proposition \ref{prop:path characterization of zero in colim}. Let us define $\xi: \varprojlim M|_{L} \rightarrow \varinjlim M|_{U}$ by $i^{-1}\circ \psi_M\circ e$. By construction, the following diagram commutes
\begin{equation}\label{eq:diagram commutes}
\begin{tikzcd}
\varprojlim M|_{L} \arrow{r}{\xi} \arrow[d, "e", "\cong"']
&\varinjlim M|_{U} \arrow[d,"i","\cong"']\\
\varprojlim M \arrow{r}{\psi_M} &\varinjlim M,
\end{tikzcd}
\end{equation}
where $\psi_M$ is the canonical limit-to-colimit map of $M$. Hence we have the fact $\rank(\psi_M)=\rank(\xi)$, which is useful for proving Theorem \ref{thm:rank is the multiplicity of full intervals}.

\subsection{Boundary cap of an interval in $\Z^2$}\label{sec:boundary cap}

Let $I\in \Int(\Z^2)$, i.e. $I$ is a finite interval of $\Z^2$ (Definition \ref{def:intervals}). By $\min (I)$ and $\max (I)$, we denote the collections of minimal and maximal elements of $I$, respectively. In other words,
\[\min (I):=\{p\in I: \mbox{there is no $q\in I$ s.t. $q<p$}\},\] \[ \max (I):=\{p\in I: \mbox{there is no $q\in I$ s.t. $p<q$}\}.
\]
Note that $\min (I)$ and $\max(I)$ are nonempty and that $\min (I)$ and $\max (I)$ respectively form an \emph{antichain} in $I$, i.e. any two different points in $\min (I)$ (or in $\max (I)$) are \emph{not} comparable.

\begin{remark}\label{rem:meet and join in Z2} 
\begin{enumerate}[label=(\roman*),itemsep=0em]
\item The least upper bound and the greatest lower bound of $p,q\in \Z^2$ are denoted by $p\vee q$ and $p\wedge q$ respectively. Let $p=(p_x,p_y)$ and $q=(q_x,q_y)$ in $\Z^2$. Then, $$p\vee q= (\max\{p_x,q_x\}, \max\{p_y,q_y\}),\hspace{3mm}p\wedge q= (\min\{p_x,q_x\}, \min\{p_y,q_y\}).$$  \label{item:meet and join in Z2 1}
\end{enumerate}
For the item below, let $I\in \Int(\Z^2)$. Notice the following: 
\begin{enumerate}[resume,label=(\roman*),itemsep=0em]
    
    \item Since $\min (I)$ is a finite antichain, we can list the elements of $\min( I)$ in ascending order of their $x$-coordinates, i.e. $\min (I):=\{p_0,\ldots,p_k\}$ and such that for each $i=0,\ldots,k,$ the $x$-coordinate of $p_i$ is less than that of $p_{i+1}$. Similarly, let $\max (I):=\{q_0,\ldots,q_\ell\}$ be ordered in ascending order of $q_j$'s $x$-coordinates. We have that $p_0\leq q_0$  (Figure \ref{fig:intervals}).\label{item:meet and join in Z2 2}
\end{enumerate}
\end{remark}

\begin{figure}
    \centering
    \includegraphics[width=0.8\textwidth]{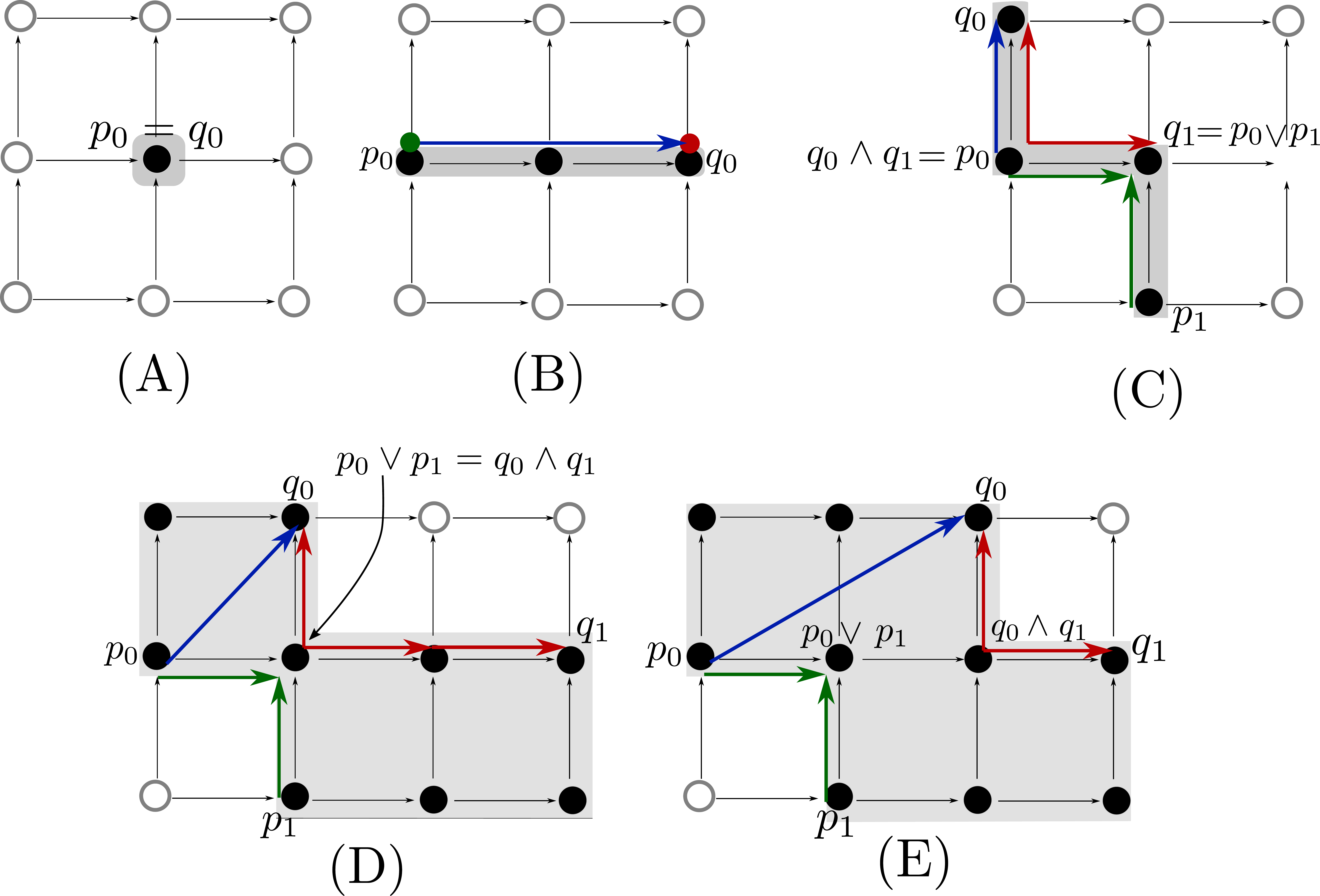}
    \caption{Five different intervals $I$ of $\Z^2$. Relations in $\min_{\ZZ}(I)$ and  $\max_{\ZZ}(I)$ are indicated by green and red arrows, respectively. The inequality $p_0\leq q_0$ is indicated by blue arrows unless $p_0=q_0$. Notice that $\partial I$, as defined in equation (\ref{eq:boundary cap}),   has cardinality $2$, $2$, $6$, $6$, $6$ in that order ((A),(B),(C),(D),(E)).}
    \label{fig:intervals}
\end{figure}

\begin{definition}[Lower and upper zigzags of an interval] 
Let $I$, $\min(I)$, and $\max(I)$ be as in Remark \ref{rem:meet and join in Z2} \ref{item:meet and join in Z2 2}. We define the following two zigzag posets (Figure \ref{fig:intervals}):
\begin{align}
    \overmin&:=\{p_0<(p_0\vee p_1)>p_1<(p_1\vee p_2)>\cdots<(p_{k-1}\vee p_k)>p_k\}\label{eq:zigzag posets}\\&=\min (I)\cup \{p_{i}\vee p_{i+1}:i=0,\ldots,k-1\}, \nonumber\\ \overmax&:=\{q_0>(q_0\wedge q_1)<q_1>(q_1\wedge q_2)<\cdots>(q_{\ell-1}\wedge q_\ell)>q_\ell\} \label{eq:zigzag posets2}\\&=\max (I)\cup \{q_{i}\wedge q_{i+1}:i=0,\ldots,\ell-1\}. \nonumber
\end{align}
\end{definition}
Note that $\overmin$ and $\overmax$ are lower and upper fences of $I$ respectively. 

For $p,q\in \Pb$, let us write $p\triangleleft q$ if $p<q$ and there is no $r\in \Pb$ such that $p<r<q$. Similarly, we write $p\triangleright q$ if $p>q$ and there is no $r\in \Pb$ such that $p>r>q$.

\begin{definition}\label{def:path}
Given a poset $\Pb$, a \textbf{path} $\Gamma$ between two points $p,q\in \Pb$ is a sequence
of points $p=p_0,\ldots, p_k=q$ in $\Pb$ such that either $p_i\leq p_{i+1}$ or
$p_i\geq p_{i+1}$ for every $i\in [1,k-1]$ (in particular, there can be a pair $i\neq j$ such that $p_i=p_j$). The path $\Gamma$ is said to be \textbf{monotonic} if $p_i\leq p_{i+1}$ for each $i$. The path $\Gamma$ is called \textbf{faithful} if either $p_i\triangleleft p_{i+1}$ or
$p_i\triangleright p_{i+1}$ for each $i$.
\end{definition}

\begin{definition}[Boundary cap of an interval]\label{def:boundary cap} We define the \textbf{boundary cap} $\partial I$ of $I\in \Int(\Z^2)$ as the path obtained by concatenating $\min_{\ZZ}(I)$ and $\max_{\ZZ}(I)$ in Eqs. (\ref{eq:zigzag posets}) and (\ref{eq:zigzag posets2}). 
\begin{equation}\label{eq:boundary cap}
    \partial I:=\hspace{5mm}\underbrace{p_k<\ (p_{k}\vee p_{k-1})>\ p_{k-1}<\cdots >\ p_0}_{2k+1\,\mbox{terms from $\overmin$}}\leq \underbrace{q_0> (q_0\wedge q_1)<\  q_1>\cdots< q_\ell}_{2\ell+1\,\mbox{terms from $\overmax$}},
\end{equation}
\end{definition}
We remark that $\partial I$ can contain multiple copies of the same point. Namely, there can be $i\in [0,k]$ and $j\in [0,\ell]$ such that either $p_i=q_j$ (Figure \ref{fig:intervals} (A)), $p_i=q_j\wedge q_{j+1}$ (Figure \ref{fig:intervals} (C)), $p_i\vee p_{i+1}=q_j$ (Figure \ref{fig:intervals} (C)), or $p_i\vee p_{i+1}=q_{j}\wedge q_{j+1}$ (Figure \ref{fig:intervals} (D)). 

Consider the following zigzag poset of the same length as $\partial I$: 
\begin{equation}\label{eq:ZZ partial I}
    \ZZ_{\partial I}:\hspace{5mm} \underbrace{\bullet_{1}<\bullet_{2}>\bullet_{3}<\cdots>\bullet_{2k+1}}_{2k+1}< \underbrace{\circ
_{1}>\circ_2<\circ_3>\cdots<\circ_{2\ell+1}}_{2\ell+1}.
\end{equation}
Still using the notation in Eqs.~(\ref{eq:boundary cap}) we have the  following order-preserving map 
\begin{equation}\label{eq:iotaI}
    \iota_I:\ZZ_{\partial I}\rightarrow I
\end{equation} whose image is $\partial I$: $\bullet_1$ is sent to $p_k$, $\bullet_2$ is sent to $p_k\vee p_{k-1}$, $\ldots$, and $\circ_{2\ell+1}$ is sent to $q_\ell$.

\subsection{Generalized rank invariant via boundary zigzags}\label{sec: generalized rank via boundary}

The goal of this section is to establish Theorem \ref{thm:rank is the multiplicity of full intervals}.

\begin{definition}\label{def:section along a path} 
Let $P$ be a poset. Let $\Gamma:p_0,\ldots,p_k$ be a path in $P$. A $(k+1)$-tuple $\bv\in \bigoplus_{i=0}^k M_{p_i}$ is called the \textbf{section of $M$ along $\Gamma$} if $\bv_{p_i}\sim\bv_{p_{i+1}}$ for each $i$ (Notation \ref{not:sim}).
\end{definition} Note that $\bv$ is not necessarily a section of the restriction  $M|_{\{p_0,\ldots,p_k\}}$  of $M$ to the subposet $\{p_0,\ldots,p_k\}\subseteq I$. Furthermore, $\Gamma$ can contain multiple copies of the same point in $P$.

\begin{example}\label{ex:section along a path}
Consider $M:\{(1,1),(1,2),(2,2),(2,1)\}(\subset \Z^2)\rightarrow \vect$ given as follows.
 \[\begin{tikzcd}
M_{(1,2)} \arrow{r}{} 
&M_{(2,2)} \\
M_{(1,1)} \arrow{r}{} \arrow{u}{} &M_{(2,1)} \arrow{u}{}
\end{tikzcd}\hspace{3mm}=\hspace{3mm}\begin{tikzcd}
\F \arrow{r}{1} 
&\F \\
\F \arrow[swap]{r}{\left(\begin{smallmatrix}1\\0\end{smallmatrix}\right)} \arrow{u}{1} &\F^2 \arrow[swap]{u}{\left(\begin{smallmatrix}1 \ 1\end{smallmatrix}\right)}\end{tikzcd}
\]
Consider the path $\Gamma:(1,1),(1,2),(2,2),(2,1)$ which contains all points in the indexing poset. Then, $\bv:=(1,1,1,(0,1))\in M_{(1,1)}\oplus M_{(1,2)}\oplus M_{(2,2)}\oplus M_{(2,1)}$ is a section of $M$ along $\Gamma$, while $\bv$ is \emph{not} a section of $M$ itself, i.e. $\bv\not\in \varprojlim M$. 
\end{example}

By Proposition \ref{prop:computation} \ref{item:colim}, we directly have:
\begin{proposition}\label{prop:path characterization of zero in colim}
Let $p,q\in P$. For any vectors $v_{p}\in M_{p}$ and $v_q\in M_q$, $[v_{p}]=[v_q]$ in\footnote{For simplicity, we write $[v_{p}]$
 and $[v_q]$ instead of $[j_p(v_p)]$ and $[j_q(v_q)]$ respectively where $j_p:M_p\rightarrow \bigoplus_{r\in P}M_r$  and  $j_q:M_q\rightarrow \bigoplus_{r\in P}M_r$ are the canonical inclusion maps.} the colimit $\varinjlim M$ if and only if there exist a path $\Gamma:p=p_0,p_1,\ldots,p_n=q$ in $P$ and a section $\bv$ of $M$ along $\Gamma$ such that $\bv_{p}=v_{p}$ and $\bv_{q}=v_q$. 
\end{proposition}


The map $\iota_I:\ZZ_{\partial I}\rightarrow I$ in Eqs.~(\ref{eq:iotaI}) induces a bijection between the sections of $M_{\partial I}$  and the sections of $M$ along $\partial I$ in a canonical way. Hence:
\begin{framed}
\begin{setup}\label{setup in section 3}\label{setup:sections}
In the rest of \S\ref{sec: generalized rank via boundary}, we fix both $I\in\Int(\Z^2)$ and  a functor $M:I\rightarrow \vect$. Each element in $\varprojlim M_{\partial I}$  is identified with the corresponding section of $M$ along $\partial I$. Also, we identify points in  (\ref{eq:boundary cap}) and (\ref{eq:ZZ partial I}) via $\iota_I$.
\end{setup}
\end{framed}

\begin{definition}[Zigzag module along $\partial I$]\label{def:zigzag module along the boundary cap}
Define the zigzag module $M_{\partial I}:\ZZ_{\partial I}\rightarrow \vect$ by $(M_{\partial I})_x:=M_{\iota_I(x)}$ for $x\in \ZZ_{\partial I}$ and $\varphi_{M_{\partial I}}(x, y):=\varphi_{M}(\iota_I(x), \iota_I(y))$ for $x\leq y$ in $\ZZ_{\partial I}$.
\end{definition}

One of our main results is the following.

\begin{theorem}\label{thm:rank is the multiplicity of full intervals}
$\rank(M)$ is equal to the multiplicity of the full interval in $\barc(M_{\partial I})$.
\end{theorem}

\begin{proof}
By Theorem \ref{thm:rk}, it suffices to show that
\[\rank(\psi_M:\varprojlim M \rightarrow \varinjlim M)=\rank(\psi_{M_{\partial I}}:\varprojlim M_{\partial I} \rightarrow \varinjlim M_{\partial I}).\]
Let $L:=\min_{\ZZ}(I)$ and $U:=\max_{\ZZ}(I)$ which are lower and upper fences of $I$ respectively. Let us define the maps $e,r,i$ and $\xi$ as described in the paragraph after Proposition \ref{prop:section extension}. Then, by Proposition \ref{prop:section extension} and the commutative diagram in (\ref{eq:diagram commutes}), it suffices to prove that the rank of $\xi$ equals the rank of $\psi_{M_{\partial I}}$.
To this end, we show that there exist a surjective linear map $f:\varprojlim M_{\partial I}\rightarrow \varprojlim M|_L$ and an injective linear map $g:\varinjlim M|_U \rightarrow \varinjlim M_{\partial I}$ such that $\psi_{M_{\partial I}}=g\circ \xi \circ f$. We define $f$ as the canonical section restriction $(\bv_q)_{q\in \partial I}\mapsto (\bv_{q})_{q\in L}$. We define $g$ as the canonical map, i.e. $[v_{q}]\mapsto [v_q]$ for any $q\in U$ and any $v_q\in M_q$. By Proposition \ref{prop:path characterization of zero in colim} and  by construction of $M_{\partial I}$, the map $g$ is well-defined. 

We now show that $\psi_{M_{\partial I}}=g\circ \xi \circ f$. Let $\bv:=(\bv_q)_{q\in \partial I}\in \varprojlim M_{\partial I}$. Then, by definition of $\psi_{M_{\partial I}}$ (Setup \ref{convention}), the image of $\bv$ via $\psi_{M_{\partial I}}$ is $[\bv_{q_0}]$ where $q_0\in U$ is defined as in Remark \ref{rem:meet and join in Z2} \ref{item:meet and join in Z2 2}. Also, we have \[\bv\stackrel{f}{\longmapsto} (\bv_{q})_{q\in L} \stackrel{\xi}{\longmapsto} [\bv_{q_0}] \big(\in \varinjlim M|_U) \stackrel{g}{\longmapsto} [\bv_{q_0}] (\in \varinjlim M_{\partial I}\big),\] which proves the equality $\psi_{M_{\partial I}}=g\circ \xi \circ f$.

We claim that $f$ is surjective. Let $r':\varprojlim M \rightarrow \varprojlim M_{\partial I}$ be the canonical section restriction map $(\bv_q)_{q\in I}\mapsto (\bv_{q})_{q\in \partial I}$. Then, the restriction $r:\varprojlim M \rightarrow \varprojlim M|_L$, can be seen as the composition of two restrictions $r=f\circ r'$.
 Since $r$ is the inverse of the isomorphism $e$ in diagram (\ref{eq:diagram commutes}), $r$ is surjective and thus so is $f$. 

Next we claim that $g$ is injective. Let $i':\varinjlim M_{\partial I}\rightarrow \varinjlim M$ be defined by $[v_q]\mapsto [v_q]$ for any $q\in \partial I$ and any $v_q\in M_q$.  By Proposition \ref{prop:path characterization of zero in colim} and  by construction of $M_{\partial I}$, the map $i'$ is well-defined. Then, for the isomorphism $i$ in diagram (\ref{eq:diagram commutes}), we have $i=i'\circ g$. This implies that $g$ is injective.
\end{proof}

\begin{remark}\label{rem:lower boundary cap}
In Definition \ref{def:boundary cap} one may consider the ``lower'' boundary cap $\widehat{\partial I}$, as an alternative to $\partial I$: \[\widehat{\partial I}: p_0<p_0\vee p_1>p_1<\cdots>p_k\leq q_{\ell}>q_{\ell}\wedge q_{\ell-1} <q_{\ell-1}>\cdots<q_0.\]
The value $\rank(M)$ also equals the multiplicity of the  full interval in the barcode of the  zigzag module induced over $\widehat{\partial I}$.
\end{remark}


By Theorem \ref{thm:rank is the multiplicity of full intervals}, we can utilize algorithms for zigzag persistence  in order to compute the generalized rank invariant and the generalized persistence diagram of any $\Z^2$-module that is obtained by applying the homology functor to a finite simplicial bifiltration consisting of $O(\NN)$ simplices over an index set of size $O(t)$. For this, we complete the boundary cap of a given interval to a faithful path (i.e. we put the missing monotonic paths between every pair of consecutive points) and then simply run a zigzag persistence algorithm, say the  $O(\NN^\omega)$ algorithm of Milosavljevic et al.~\cite{milosavljevic2011zigzag}, on the filtration
restricted to this path.

\begin{remark}
To compute  $\intdgm(M)(I)$ by the formula in~(\ref{eq:formula}), one needs to consider terms whose number depends exponentially on the number of neighbors of $I$.
However, for any interval that has at most $O(\log \NN)$ neighbors, we have $2^{O(\log \NN)}=\NN^c$ terms for some
constant $c>0$. It follows that using $O(\NN^{\omega})$ zigzag persistence algorithm for computing generalized ranks, we obtain
an $O(\NN^{\omega+c})$ algorithm for computing generalized persistence diagrams of intervals that have at most $O(\log \NN)$ neighbors.
\end{remark}

\section{Computing intervals and detecting interval decomposability}\label{sec:computing interval summands and detecting interval decomposability}

When a persistence module $M$ admits a summand $N$ that is isomorphic to an interval module, $N$ will be called an \textbf{interval summand} of $M$. In this section, we apply Theorem \ref{thm:rank is the multiplicity of full intervals} 
for computing generalized rank via zigzag 
to different problems that ask to find interval summands of an input finite $\Z^2$-module: Problems \ref{prob:isinterval}, \ref{prob:compute summands of interval decomposable modules}, and \ref{prob:interval decomposability}.

Let $K$ be a finite abstract simplicial complex and let $\sub(K)$ be the poset of all subcomplexes of $K$, ordered by inclusion. Given any poset $\Pb$, an order-preserving map $\Fcal:\Pb\rightarrow \sub(K)$ is called a simplicial filtration (of $K$).

\begin{framed}
\begin{setup}\label{setup}
Throughout \S\ref{sec:computing interval summands and detecting interval decomposability},  $\mathcal F$ denotes a bifiltration of a simplicial complex $K$ defined over an \emph{interval} $\Pb\in \Int(\Z^2)$. Let $t:=\max(|K|,|P|)$ denote the maximum of the number of simplices 
in $K$ and the number of points in $\Pb$.
By $M_{\mathcal F}: \Pb\rightarrow \vect$ we denote the module induced by $\mathcal F$ through  the homology functor with coefficients in the field $\F$.
\end{setup}
\end{framed}

\paragraph{Computing the dimension function.}\label{sec:setup and dimensions}

In all algorithms below, we utilize a subroutine
{\sc Dim}$({\mathcal F},\Pb)$, which computes the dimension of the vector space $(M_{\mathcal F})_p$ for every $p\in \Pb$.
\begin{proposition}\label{prop:dimension function}
{\sc Dim}$({\mathcal F},\Pb)$ can be executed in $O(\NN^3)$ time.
\label{prop:dimension}
\end{proposition}

\begin{proof}
To implement {\sc Dim},  
we maintain a set $C$ which we initialize to be the empty set. 
We consider each $p\in \min(\Pb)$ iteratively and proceed as follows. We 
reach all points $q\in p^{\uparrow}\backslash C$ over faithful paths. 
These paths form a directed tree $Q$ rooted at $p$  (where the arrows do not zigzag). We  run the standard persistence algorithm on the filtration $\mathcal F$ restricted
to $Q$ with 
a slight modification so that
the branchings in the tree are taken care of.  We traverse the tree $Q$ in a depth first manner and each time we come to a branching node, we start with the matrix which was computed during the previous visit of this node. This means that we leave a copy of the reduced matrix at each node that we traverse. 
This entire process takes time $O(\NN_p^3)$ where $\NN_p$ equals the maximum of the cardinality of $p^\uparrow\backslash C$ and the number of simplices
in the filtration over $p^\uparrow\backslash C$. This is because of the following.
At each node $q$ we spend $O(\NN_p^2)$ time to copy
a matrix of size $O(\NN_p)\times O(\NN_p)$. Also we spend $O(t_p^2s_q)$ time to
reduce $s_q$ columns in a matrix of size $O(\NN_p)\times O(\NN_p)$ corresponding to the new $s_q$ simplices introduced at the node $q$. This incurs a total cost of $\sum_q O(t_p^2)s_q=O(t_p^3)$ because $\sum_q s_q\leq  t_p$. We then update $C \leftarrow  C\cup (p^\uparrow\setminus C)$.
Taking the sum over all minimal points we see that the overall cost is of order $\sum_{p} O(\NN_p^3)=O(\NN^3)$ since $\sum_p \NN_p=\NN$. 
\end{proof}

\subsection{Detecting interval modules}
We consider the following problem. 
\begin{problem}\label{prob:isinterval}
Determine whether $M_{\mathcal{F}}$ is isomorphic to the direct sum of a certain number of copies of $\I_\Pb$ and if so, report the number of such copies.
\end{problem}

Algorithm {\sc IsInterval} solves Problem~\ref{prob:isinterval}.
 The correctness of the algorithm follows from Proposition~\ref{prop:IsInterval}. Below, for an interval
$I\in \Int(\Z^2)$ and for $m\in \Zplus$ we define $\I_I^{m}:=\underbrace{\I_I\oplus\I_I\oplus\cdots \oplus \I_I}_{m}$. In particular, $\I_I^0$ is defined to be the trivial module. Let us recall that $(M_{\mathcal F})_{\partial \Pb}$ denotes the zigzag module along the boundary cap $\partial \Pb$ (Definition \ref{def:zigzag module along the boundary cap}).\\

\noindent
{\bf Algorithm} {\sc IsInterval}($\cal F$, $\Pb$)
\begin{itemize}
    \item Step 1. Compute zigzag barcode $\barc((M_{\mathcal F})_{\partial \Pb})$ 
    and let $m$ be the multiplicity of the full interval.
    \item Step 2. Call {\sc Dim}$({\mathcal F},\Pb)$ (Computes $\dim (M_{\mathcal F})_p$ for every $p\in \Pb$)
    \item Step 3. If $\dim (M_{\mathcal F})_p==m$ for each point $p\in \Pb$ return $m$, otherwise return $0$ indicating $M_{\mathcal F}$ has a summand which is not an interval
    module supported over $\Pb$.
\end{itemize}

\begin{proposition}
Assume that a given $M:\Pb\rightarrow \vect$ has the indecomposable decomposition $M\cong\bigoplus_{i=1}^m M_i$. Then, every summand $M_i$ is isomorphic to the interval module $\I_{\Pb}$  if and only if $\intrk(M)(\Pb)=\dim M_p=m$ for all $p\in \Pb$.
\label{prop:IsInterval}
\end{proposition}

\begin{proof}By Theorem \ref{thm:rk}, the forward direction is straightforward. Let us show the backward direction. Since $\intrk(M)(\Pb)=m$, by Theorem \ref{thm:rk}, we have that 
$M\cong \I_{\Pb}^m \oplus M'$ 
where $M'$ admits no summand that is isomorphic to $\I_\Pb$. Fix any $p\in \Pb$. Then, $m=\dim(M_p)=\dim ((\I_{\Pb}^m )_p)+\dim(M'_p)=m+\dim(M'_p)$, and hence $\dim(M'_p)=0$. Since $p$ was chosen arbitrarily, $M'$ must be trivial.
\end{proof}

\begin{proposition}\label{prop:IsInverval complexity}
Algorithm {\sc IsInterval} can be run in $O(\NN^3)$ time. 
\end{proposition}

\begin{proof}
As we commented earlier, Step 1 computing the  zigzag barcode $\barc((M_{\Fcal})_{\partial \Pb})$ can be implemented to run in $O(\NN^\omega)$ time because
it runs on a filtration comprising $O(\NN)$ simplices over an index set of size $O(\NN)$ . 
Step 2 takes $O(\NN^3)$ time (Proposition~\ref{prop:dimension}).
Step 3 takes only time $O(\NN)$. 
This implies that the overall complexity is $O(\NN^3)$ given that $\omega<2.373$.
\end{proof}

\subsection{Interval decomposable modules and its summands}\label{sec:Interval decomposable modules and its summands}
Setup \ref{setup} still applies in \S\ref{sec:Interval decomposable modules and its summands}. Next, we consider the problem of computing all indecomposable summands of $M_\Fcal$ \emph{under the assumption that $M_\Fcal$ is interval decomposable} (Definition \ref{def:interval decomposable}). 

\begin{problem}\label{prob:compute summands of interval decomposable modules}
Assume that $M_{\Fcal}:\Pb\rightarrow \vect$ is interval decomposable. Find $\barc(M_{\Fcal})$.
\label{prob:intervaldecomp}
\end{problem}
We present algorithm {\sc Interval} to solve Problem~\ref{prob:intervaldecomp} in $O(\NN^{\omega+2})$ time. This
algorithm is eventually used to detect whether a given module is interval decomposable or not (Problem \ref{prob:interval decomposability}).
Before describing {\sc Interval}, we first describe another algorithm {\sc TrueInterval}. The outcomes of both {\sc Interval} and {\sc TrueInterval} are the same as the barcode of $M_\Fcal$ in Problem \ref{prob:intervaldecomp} (Propositions \ref{prop:intervals} and \ref{prop:two algorithm coincides}). Whereas {\sc TrueInterval} is more intuitive, real implementation is accomplished via {\sc Interval}.

\begin{definition}\label{def:maximal support interval}
Let $\mathcal{I}(M_{\Fcal}):=\{I\in\Int(\Pb):\,\intrk(M_{\Fcal})(I)>0\}$. We call $I\in\mathcal{I}(M_{\Fcal})$  \textbf{maximal} if there is no $J\supsetneq I$ in $\Int(\Pb)$ such that $\intrk(M_{\Fcal})(J)$ is nonzero.
\end{definition}

\begin{proposition}\label{prop:maximal intervals}\label{prop:maxI}
Assume that $M_\Fcal$ is interval decomposable and let $I\in \mathcal{I}(M_{\Fcal})$ be maximal. Then, $I$ belongs to $\barc(M_{\Fcal})$ and the multiplicity of $I$ in $\barc(M_{\Fcal})$ is equal to $\intrk(M_{\Fcal})(I)$.
\end{proposition}

\begin{proof}
By assumption, all summands in the sum $\displaystyle \sum_{\substack{A\subseteq \nbd_\mathbb{I}(I)\cap P\\A\neq\emptyset}}(-1)^{\abs{A}}\intrk(M_{\mathcal F})\bigg(\overline{I\cup A}\bigg)$ corresponding to the second term  of (\ref{eq:formula}) are zero. Hence, $\intdgm(M_{\mathcal F})(I)=\intrk(M_{\mathcal F})(I)>0$. Since $M_{\mathcal F}$ is interval decomposable, by Theorem \ref{thm:dgm generalizes barc}, $\intdgm(M_{\mathcal F})(I)$ is equal to the multiplicity of $I$ in $\barc(M_{\mathcal F})$. Therefore, not only does $I$ belong to $\barc(M_{\mathcal F})$, but also the value $\intrk(M_{\mathcal F})(I)$ is equal to the multiplicity of $I$ in $\barc(M_{\mathcal F})$. 
\end{proof}

The following proposition is a corollary of Proposition \ref{prop:maximal intervals}. 
\begin{proposition}\label{prop:MaxI}
Assume that $M_\Fcal$ is interval decomposable and let $I\in \mathcal{I}(M_{\Fcal})$ be maximal. Let  $\mu_I:=\intrk(M_{\Fcal})(I)$. Then,
$M_\Fcal$ admits a summand $N$ which is isomorphic to $\I_I^{\mu_I}$. 
\end{proposition}

Let us now describe a procedure {\sc TrueInterval} that outputs all
indecomposable summands of a given interval decomposable module. For computational
efficiency, we will implement {\sc TrueInterval} differently. Let $M:=M_{\mathcal F}$ . First we compute $\dim M_p$ for every
point $p\in \Pb$. Iteratively, we choose a point $p$ with $\dim M_p\not= 0$
and compute a maximal interval $I\in\mathcal{I}(M)$ containing $p$. 
Since $M$ is interval decomposable, by Propositions~\ref{prop:maxI} and \ref{prop:MaxI} we have that $I\in \barc(M)$ and that there is a summand $N\cong \I_{I}^{\mu_I}$ of $M$. Consider the quotient module $M':= M/N$. Clearly, this `peeling off' of
$N$ reduces the total dimension of the input module. Namely,  $\dim M'_p= \begin{cases}
\dim M_p-\mu_I,&p\in I \\ \dim M_p,&p\notin I. \end{cases}$ We
continue the process by replacing $M$ with $M'$ until
there is no point $p\in \Pb$ with $\dim M_p\not = 0$ (note that $M'$ is interval decomposable by Proposition \ref{prop:quotienting does not change interver decomposabilty}). Since $\dim M:=\sum_{p\in \Pb}\dim M_p$ is finite, this process terminates in finitely many steps. By Propositions  \ref{prop:quotienting does not change interver decomposabilty} and \ref{prop:MaxI}, the outcome of {\sc TrueInterval} is a list of all intervals in $\barc(M)$ with
accurate multiplicities: 

\begin{proposition}\label{prop:intervals}
Assume that $M_{\Fcal}$ is interval decomposable. Let $I_i$, $i=1,\ldots, k$ be the intervals computed by {\sc TrueInterval}. For each $i=1,\ldots, k$,
let $\mu_{I_i}:=\intrk(M_\Fcal)({I_i})$.
Then, we have $M_\Fcal \cong \bigoplus_{i=1}^k \I_{I_i}^{\mu_{I_i}}$.
\end{proposition}

Next, we describe an algorithm {\sc Interval}
that simulates {\sc TrueInterval}
while avoiding explicit quotienting of $M_{\mathcal F}$ by its summands. 

We associate a number $d(p)$ and a list ${\tt list}(p)$ of identifiers of intervals $I\subseteq P$ to each point $p\in \Pb$. The number $d(p)$ equals the original
dimension of $(M_{\mathcal F})_p$ minus the
number of intervals peeled off so far (counted with their multiplicities) 
which contained $p$.
It is initialized to $\dim (M_{\mathcal F})_p$.
Each time we compute
a maximal interval $I\in \mathcal{I}(M_\Fcal)$ with multiplicity $\mu_I$ that contains $p$, we update 
$d(p):=d(p)-\mu_I$ keeping track of how many more intervals containing $p$ would
{\sc TrueInterval}  still be peeling off.

With each interval $I$ that is output, we associate an identifier 
$\mathrm{id}(I)$.
The variable ${\tt list}(p)$ maintains the set of identifiers of the intervals
containing $p$ that have been output so far. While searching for a maximal
interval $I$, we maintain a variable ${\tt list}$ for $I$
that contains the set of identifiers common to all points in $I$. Initializing
${\tt list}$ with ${\tt list}(p)$ of the initial point $p$,
we update it as we explore expanding $I$.
Every time we augment $I$ with a new point
$q$, we update ${\tt list}$ by taking its intersection with the
set of identifiers ${\tt list}(q)$ associated with $q$.

We assume a routine {\sc Count} that takes a list as input and gives the total
number of intervals counted with their multiplicities whose identifiers are in the
list. This means that if ${\tt list}=\{\mathrm{id}(I_1),\ldots,\mathrm{id}(I_k)\}$,
then {\sc Count}($\tt list$) returns the number $c:=\sum \mu_{I_1}+\cdots+\mu_{I_k}$. 

Notice that, while searching for a
maximal interval starting from a point, we keep considering the original
given module $M_{\mathcal F}$ since we do not implement the true `peeling' (i.e. quotient $M_{\Fcal}$ by a submodule).
However, we modify the condition for checking the maximality of an interval $I$.
We check whether 
$\intrk(M_{\Fcal})(I)>c$, that is, whether the generalized
rank of $M_{\Fcal}$ over $I$ is larger than the total number of
intervals containing $I$ that would have been peeled off so far by {\sc TrueInterval}.
This idea is implemented in the following algorithm.\\

\noindent
{\bf Algorithm} {\sc Interval} ($\mathcal F$, $\Pb$)\label{algorithm:interval}
\begin{itemize}
    \item Step 1. Call {\sc Dim}($\mathcal F$,$\Pb$) 
    and set $d(p):=\dim (M_{\mathcal F})_p$; ${\tt list}(p):=\emptyset$ for every $p\in \Pb$
    \item Step 2. While there exists a $p\in \Pb$ with $d(p)> 0$ do
    \begin{itemize}
        \item Step 2.1 Let $I:=\{p\}$; ${\tt list}:={\tt list}(p)$; unmark every $q\in \Pb$
        \item Step 2.2 If there exists unmarked $q\in \intnbd(I)$  then 
        \begin{enumerate}
            \item[i.] ${\tt templist}:={\tt list}\cap {\tt list}(q)$; $c:=${\sc Count}($\tt templist$)
           \item[ii.] If $\intrk(M_{\mathcal F})(I\cup\{q\})>c$ then\footnote{to check $\intrk(M_{\mathcal F})(I\cup\{q\})>c$, we invoke Theorem \ref{thm:rank is the multiplicity of full intervals} and run the zigzag persistence algorithm described beneath Remark \ref{rem:lower boundary cap}. For efficiency, one can use zigzag update algorithm in~\cite{DH21}.} 
             mark $q$; set $I:=I\cup \{q\}$; ${\tt list}:={\tt list}\cap {\tt list}(q)$ 
            \item[iii.] go to Step 2.2
        \end{enumerate}
        \item Step 2.3 Output $I$ with multiplicity $\mu_I:=\intrk(M_{\mathcal F})(I)-c$
        \item Step 2.4 For every $q\in I$ set $d(q):=d(q)-\mu_I$ and 
        ${\tt list}(q):={\tt list}(q)\cup \{\mathrm{id}(I)\}$
    \end{itemize}
\end{itemize}
The output of {\sc Interval} can be succinctly described as:\\
{\bf Output:} $\{(I_i,\mu_{I_i}):i=1,\ldots,k\}$ where $I_i\in\Int(\Pb)$ and $\mu_i$ is a positive integer for each $i$.

\begin{remark} For each $p\in \Pb$, $\dim M_p$ coincides with $\sum_{I_i\ni p} \mu_i$. 
\end{remark}

We will show that if $M_\Fcal$ is interval decomposable, then the output of {\sc Interval} coincides with the barcode 
of $M_{\Fcal}$ (Propositions \ref{prop:intervals} and \ref{prop:two algorithm coincides}). 

\begin{example}[{\sc Interval} with interval decomposable input]\label{ex:intervals}
Suppose that $M_{\Fcal}\cong \I_{I_1}\oplus \I_{I_2}\oplus\I_{I_3}$ as depicted in Figure \ref{fig:example_algorithm1} (A). The algorithm {\sc Interval} yields $\{(I_1,1),(I_2,1),(I_3,1)\}$. In particular, since $I_1\supset I_2 \supset I_3$, {\sc Interval} outputs $(I_1,1)$, $(I_2,1)$, and $(I_3,1)$ in order, as depicted in Figure \ref{fig:example_algorithm2} (A).
\end{example}

\begin{proof}[Details for Example \ref{ex:intervals}]
We illustrate how the barcode of $M_{\Fcal}$ is obtained as the output of {\sc Interval}. In Step 1, for $p=(2,2)$, we have $d(p)=3$ and ${\tt list}(p)=\emptyset$. For the first round of the while loop in Step 2, note that $q=(2,1)$ belongs to $\intnbd(\{p\})$ and thus $q$ can be added to $I=\{p\}$. After several rounds of the while loop, we obtain $I_1$ with $\mu_{I_1}=1$, and for every $q\in I_1$ (including $p=(2,2)$), $d(q)$ is decreased by $1$. Next, suppose that another search begins at $I=\{p\}$ and
it tries to include $q=(2,1)$. We obtain $\intrk(M_{\Fcal})(I\cup\{q\})=1$ as in the previous round, but now {\sc Count}$({\tt list}(p)\cap {\tt list}(q))$ also returns $1$ in Step 2.2(i) because the previously detected interval $I_1$ contains both $p$ and $q$. 
Then, the test $\intrk(M_{\Fcal})(I\cup\{q\})> (c=1)$ in Step 2.2(ii) fails and
the search proceeds with other points. Again, after several iterations of successful
and unsuccessful attempts to expand $I$, we obtain the interval $I_2$ with $\mu_{I_2}=1$. After obtaining $(I_2,\mu_2)$, only $p=(2,2)$ has $d(p)=1$ because
of which another round of while loop starting at $p$ outputs the interval $I_3=\{p\}$.
\end{proof}

\begin{example}[{\sc Interval} with non-interval-decomposable input]\label{ex:intervals2}Suppose that $N:=M_{\Fcal}\cong N'\oplus \I_{I_2}$ as depicted in Figure \ref{fig:example_algorithm1} (B). $N'$ is an indecomposable module that is not an interval module. One possible final output of {\sc Interval} is $\{(J_1,1),(J_2,1),(J_3,1)\}$ as depicted in Figure \ref{fig:example_algorithm2} (B). Note however that, depending on the choices of $p$ in Step 2 and the neighbors $q$ in Step 2.2, the final outcome can be different.
\end{example}

\begin{proof}[Details for Example \ref{ex:intervals2}]
In Step 1 of {\sc Interval}, for $p=(2,2)$, we have $d(p)=3$ and ${\tt list}(p)=\emptyset$. In Step 2.2, $q=(1,2)$ belongs to $\intnbd(\{p\})$ and thus $q$ can be added to $I=\{p\}$. Once $I$ becomes $\{p,q\}$, after multiple iterations within Step 2.2, $I$ will expand to $J_1$ in Figure \ref{fig:example_algorithm2} (B). Therefore, the interval $J_1$ and $\mu_{J_1}=1$ will be a part of the output. By continuing this process, one possible final output is $\{(J_1,1),(J_2,1),(J_3,1)\}$ as depicted in Figure \ref{fig:example_algorithm2}.
\end{proof}

\begin{figure}[h]
    \centering
    \includegraphics[width=0.9\textwidth]{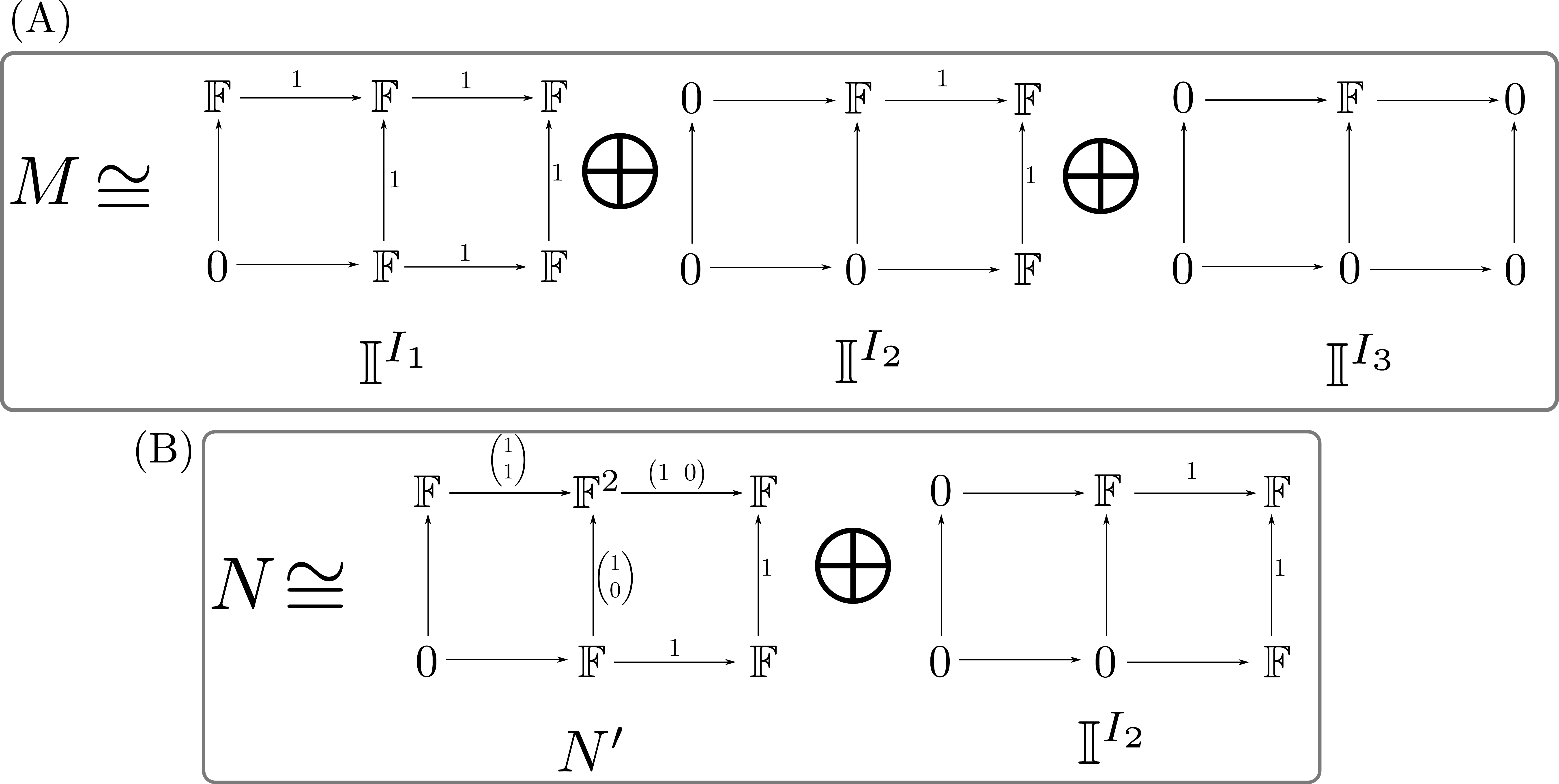}
    \caption{Modules $M,N:\{1,2,3\}\times\{1,2\}\rightarrow \vect$. $M$ is interval decomposable, but $N$ is not.}
    \label{fig:example_algorithm1}
\end{figure}

\begin{figure}
    \centering
    \includegraphics[width=0.6\textwidth]{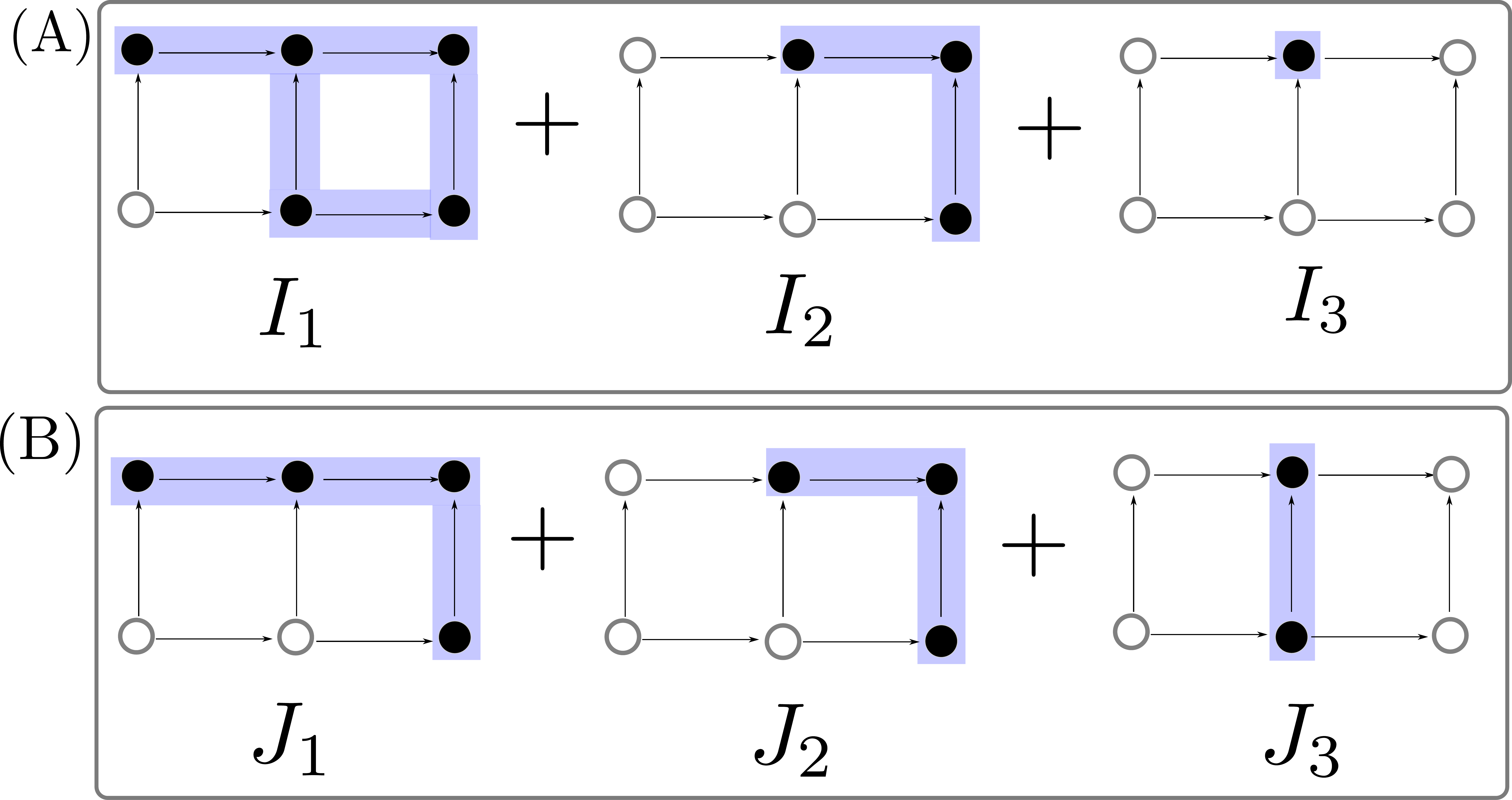}
    \caption{An illustration for Examples \ref{ex:intervals} and \ref{ex:intervals2}} 
    \label{fig:example_algorithm2}
\end{figure}

\begin{proposition}\label{prop:two algorithm coincides}
If $M_{\mathcal F}$ is interval decomposable,
{\sc Interval}$(\mathcal F$, $\Pb)$ computes an interval in $\barc(M_{\Fcal})$ if and only if {\sc TrueInterval}$(\mathcal F$, $\Pb)$ computes it with the same multiplicity.
\end{proposition}
\begin{proof}
(`if'): We induct on the list of intervals in the order they are computed
by {\sc TrueInterval}. We prove two claims by induction: (i) {\sc TrueInterval}
can be run to explore the points in $\Pb$ in the same order as {\sc Interval}
while searching for maximal intervals, (ii) if $I_i$, $i=1,\cdots, k$, are 
the intervals computed by {\sc TrueInterval} with this chosen order, 
then {\sc Interval} also outputs these intervals with the same multiplicities. 
Clearly, for $i=1$, {\sc Interval} computes the maximal
interval on the same input module $M_{\mathcal F}$ as {\sc TrueInterval} does.
So, clearly, {\sc TrueInterval} can be made to explore $\Pb$ as {\sc Interval}
does and hence their outputs are the same. 
Assume inductively that the hypotheses
hold for $i\geq 1$. Then, {\sc TrueInterval} 
operates next on the module $M_{i+1}:=M_{\mathcal F}/(\I_{I_1}^{\mu_{I_1}}\oplus\cdots\oplus \I_{I_i}^{\mu_{I_i}})$ (here each $\I_{I_i}^{\mu_{I_i}}$ stands for a summand of $M_{\Fcal}$ that is isomorphic to $\I_{I_i}^{\mu_{I_i}}$ by Proposition \ref{prop:MaxI}). We let {\sc TrueInterval} explore $\Pb$ in the same way
as {\sc Interval} does. This is always possible because the outcome of the
test for exploration remains the same in both cases as we argue.
The variable $d(p)$ at this point has the value $\dim (M_{i+1})_p$ and thus
both {\sc TrueInterval} and {\sc Interval} can start exploring from the
point $p$ if $d(p)>0$. So, we 
let {\sc TrueInterval} compute the next maximal interval $I_{i+1}$ starting from the point $p$ 
if {\sc Interval} starts from $p$.

Now, when {\sc Interval} tests for a point $q$ to expand the
interval $I$, we claim that the result would be the same if {\sc TrueInterval}
tested for $q$.
First of all, the condition whether $I\cup\{q\}$ is an interval or not does not
depend on which algorithm we are executing. Second, 
the list supplied to {\sc Count} in Step 2.2 (i) exactly
equals the list of intervals containing $I\cup\{q\}$
that {\sc Interval} has already output.
By the inductive hypothesis, this list
is exactly equal to the list of intervals that {\sc TrueInterval} had already `peeled off'. Therefore, the test
$\intrk(M_{\mathcal F})(I\cup\{q\})>c$ that {\sc Interval} performs in Step 2.2 (ii) is exactly
the same as the test $\intrk(M_{i+1})(I\cup\{q\})>0$ that {\sc TrueInterval} would
have performed for the module $M_{i+1}$. This establishes that
{\sc Interval} computes the same interval $I_{i+1}$ with the same multiplicity
as {\sc TrueInterval} would have
computed on $M_{i+1}$ using the same order of exploration as the inductive hypothesis claims.

(`only if'): We already know that {\sc Interval} computes all intervals
that {\sc TrueInterval} computes. We claim that it does not compute any other
interval. For intervals computed by {\sc TrueInterval}, one has that
$\dim (M_{\mathcal F})_p=\sum_i(\I_i^{\mu_{I_i}})_p$, or equivalently
$\dim (M_{\mathcal F})_p=\sum_{I\in\{I_1,\ldots,I_k\} s.t. I\ni p} \mu_I$. The algorithm
{\sc Interval} decreases the variable $d(p)$ exactly by the amount
on the right-hand side of the equation and is intialized to $\dim(M_{\mathcal F})_p$. Therefore,
every $d(p)$ becomes equal to $\dim(M_{\mathcal F})_p$ after {\sc Interval} computes the intervals that
{\sc TrueInterval} computes. The condition $d(p)>0$ in the
{\it while} loop prohibits {\sc Interval} to compute any other interval.
\end{proof}

\begin{proposition}\label{prop:Interval time complexity}
{\sc Interval}$(\mathcal F$, $\Pb)$ runs in $O(\NN^{\omega +2})$ time.
\label{prop:timeInterval}
\end{proposition}

\begin{proof}
Each iteration in the {\it while} loop executes a traversal of the graph underlying
the poset $\Pb$ starting from a point $p$. Each time, we reach a new point $q$ in this
traversal, we execute a zigzag persistence computation on the boundary cap $\partial (I\cup \{q\})$.
This means, the number of times a zigzag persistence is computed equals the number of times
a point in the poset is considered by the {\it while} loop. We claim that this number is $O(\NN^2)$.
Each time a point $p$ is considered by the {\it while} loop, either we include it in an interval
that is output (a successful attempt), or we don't include it in the expansion of the current
interval (unsuccessful attempt) and $q$ appears as a point in the neighborhood of an interval that is output.
The number of times a point is involved in a successful attempt is at most $\NN$ because a point can be contained in at most $\NN$ intervals ($\dim (M_{\mathcal F})_p\leq \NN$). Similarly, the number of times a point is involved in an unsuccessful attempt is at most $4\NN$ because the point
can be in the neighborhood of at most $4\NN$ intervals (at most $\NN$ intervals for each of its $\leq 4$ neighbors).
Therefore, each point participates in at most $O(\NN)$ computations of zigzag persistence 
over the entire {\it while} loop. Each zigzag computation takes time $O(\NN^{\omega})$ since the 
filtration ${\mathcal F}|_{\partial (I\cup \{q\})}$ restricted to the boundary cap has  length at most $\NN$ comprising at most $t$ simplices. 
It follows that the total cost due to zigzag persistence computation
is bounded by $O(\NN^{\omega+2})$.

Now, we analyze the cost of maintaining the lists with each point and 
with the intervals under construction.
Notice that $\dim (M_{\mathcal F})=\sum_{p\in \Pb} \dim (M_{\mathcal F})_p=O(\NN^2)$ because there are at most $t$ points in
$\Pb$ with $\dim (M_{\mathcal F})_p\leq \NN$ for each $p\in \Pb$ since $\mathcal F$ has at most $\NN$ simplices. Each {\it while} loop iteration maintains a global list, calls {\sc Count} on this list,
and updates ${\tt list}(q)$ for some points $q\in \Pb$. 
The cost of this counting and updates cannot be more than the order of the final total size $\sum_p {\tt list}(p)$ of the lists, which in turn
is no more than $\dim (M_{\mathcal F})=O(\NN^2)$. Thus, over the entire {\it while} loop
we incur $O(\NN^4)$ cost for maintenance of the lists and for the counting based on them. 
Thus, we have a worst-case complexity of $O(\NN^4+\NN^{\omega+2})=O(\NN^{\omega+2})$ because it
is known that $\omega\geq 2$.
\end{proof}

\subsection{Interval decomposability}\label{sec:Interval decomposability}
Setup \ref{setup} still applies in \S\ref{sec:Interval decomposability}. We consider the following problem.
\begin{problem}\label{prob:interval decomposability}
Determine whether the module $M_{\mathcal F}$ is interval decomposable or not.
\label{prob:decomposability}
\end{problem}

If the input module $M_{\mathcal F}$ is interval decomposable, then the algorithm {\sc Interval} computes all intervals in the barcode. However, if the module $M_{\mathcal F}$ is not interval decomposable, then the algorithm is not guaranteed to output all interval summands.
We show that {\sc Interval} still can be used to solve Problem~\ref{prob:decomposability}.
For this we test whether each of the output intervals $I$ with
multiplicity $\mu_I$ indeed supports a summand $N\cong \I_I^{\mu_I}$ of $M_{\mathcal F}$. 

To do this we run  Algorithm 3 in Asashiba et al.~\cite{asashiba2018interval} for each of the output intervals of {\sc Interval}. Call this algorithm {\sc TestInterval} which with
an input interval $I$, returns $\mu_I>0$ if the module $\I_I^{\mu_I}$ is a summand
of $M$ and $0$ otherwise. 

For each of the intervals $I$ with multiplicity $\mu_I$ returned by {\sc Interval}($\mathcal F$, $\Pb$)
we test whether {\sc TestInterval}($I$) returns a non-zero $\mu_I$.
The first time the test fails, we declare that $M_{\mathcal F}$ is not interval decomposable. 
This gives us a polynomial time algorithm (with complexity $O(\NN^{3\omega+2})$)
to test whether a module induced by a given bifiltration is interval decomposable or not.
It is a substantial improvement over the result of Asashiba et al.~\cite{asashiba2018interval} who gave an algorithm for tackling the same problem. Their algorithm cleverly enumerates  the 
intervals in the poset to test, but still tests exponentially many of them and hence may run in time that is exponential in $\NN$. Because of our algorithm {\sc Interval},
we can do the same test but only on polynomially many intervals. \\

\noindent
{\bf Algorithm} {\sc IsIntervalDecomp} ($\mathcal F$, $\Pb$)\label{algorithm:isintervaldecomp}
\begin{itemize}
    \item Step 1. ${\mathcal I=\{(I_i,\mu_{I_i})\}}\leftarrow$ {\sc Interval}($\mathcal F$, $\Pb$)
    \item Step 2. For every $I_i\in {\mathcal I}$ do
    \begin{itemize}
        \item Step 2.1 $\mu\leftarrow$ {\sc TestInterval}($M_{\mathcal F}$,$I_i$) 
        \item Step 2.2 If $\mu\not = \mu_{I_i}$ then output false; quit
    \end{itemize}
    \item Step 3. output true
\end{itemize}
\begin{proposition}\label{prop:complexity for testing interval decomposability}
{\sc IsIntervalDecomp}$(\mathcal F$, $\Pb)$ returns true if and only if $M_{\mathcal F}$ is interval decomposable. It takes $O(\NN^{3\omega+2})$ time.
\end{proposition}
\begin{proof}
By the contrapositive of Proposition~\ref{prop:intervals}, 
if for any of the computed interval(s) $I_i$, $i=1,\cdots,k$ by {\sc Interval}, $\I_{I_i}^{\mu_{I_i}}$
is not a summand of $M_{\mathcal F}$, then $M_{\mathcal F}$ is not interval decomposable. On the other hand, if 
every such interval module is a summand of $M_{\mathcal F}$, then we have that $M_{\mathcal F}\cong \bigoplus_{i=1}^k \I_{I_i}^{\mu_{I_i}}$ because $\dim (M_{\Fcal})_p= \sum_i^k \dim (\I_{I_i}^{\mu_{I_i}})_p$ for every $p\in \Pb$.

\paragraph{Time complexity:} By Proposition~\ref{prop:timeInterval}, Step 1 runs in time $O(\NN^{\omega+2})$.
We claim that $\dim(M_{\Fcal})=O(\NN^2)$ (see Proof of Proposition~\ref{prop:Interval time complexity}).
Therefore, {\sc Interval} returns at most $O(\NN^2)$ intervals.
According to the analysis in Asashiba et al.~\cite{asashiba2018interval},
each test in Step 2.1 takes $O(((\dim M_{\Fcal})^\omega + \NN)\NN^{\omega})=O(\NN^{3\omega})$ time and thus $O(\NN^{3\omega+2})$ in total over all $O(t^2)$ tests which
dominates the time complexity of {\sc IsIntervalDecomp}.
\end{proof}

\subsection{{\sc Interval} produces partial sections of indecomposable summands}

The algorithm {\sc Interval} produces all intervals of an input interval decomposable module. A natural question is what does the algorithm {\sc Interval} return on a module that is not interval decomposable (Figure~\ref{fig:example_algorithm2} (B)). We show that the intervals returned by the algorithm support ``partial" sections of indecomposable summands:

\begin{proposition}
Let $\mathcal I$ be the set of intervals computed by {\sc Interval}($\mathcal F$,$P$).
Then, for every $I\in {\mathcal I}$, there exists a section supported over $I$ of an indecomposable summand of $M_{\mathcal{F}}$. Moreover, for every $p\in P$, $\dim (M_{\mathcal F})_p=\sum_{I\in {\mathcal I}} \dim(\mathbb{I}_{I}^{\mu_I})_p$ if $I$ is computed with multiplicity $\mu_I>0$.
\end{proposition}
The above result follows from the proposition below since {\sc Interval} outputs an interval $I$ only if $\rk(M_{\mathcal F})(I)> 0$. 

\begin{proposition}\label{prop:rank and section}
Let $P$ be a finite connected poset. Let $M$ be a $P$-module with an indecomposable decomposition $M\cong \bigoplus_{j\in L} M_j$ for some finite set $L$. Let $I\in \Int(P)$. If $\rk(M)(I)=c>0$, there are $j_1,\ldots,j_c\in L$ such that for each $t=1,\ldots, c$, there exists a section $\bv$ of $M_{j_t}|_I$ that is fully supported, i.e. $\bv_p\neq 0$ for all $p\in I$.\footnote{We remark that these sections can be further extended using the structure maps to submodules which may have larger supports than the original sections.}
\end{proposition}

\begin{proof}By the assumption, we have
\[M|_I\cong \left(\bigoplus_{j\in L}M_j\right)\Big|_I\cong \bigoplus_{j\in L}M_j|_I.\]
Note that $M_j|_I$ can be decomposable for each $j$. Let $M_j|_I$ have an indecomposable decomposition \[M_j|_I\cong \bigoplus_{k\in K_j}M_{jk},\]
i.e. each $M_{jk}$ is an indecomposable $I$-module. Then $M|_I$ has an indecomposable decomposition \[M|_I\cong\bigoplus_{j\in L}\bigoplus_{k\in K_j}M_{jk}.\]
 Since $I$ is a finite poset and $M_p$ is finite dimensional for each $p\in \Pb$, we have that $\varprojlim M|_I$ is finite dimensional. Hence the notions of direct product and direct sum coincide in the category of $I$-indexed modules. This implies that $\varprojlim (N_1\oplus N_2)\cong \varprojlim N_1 \oplus \varprojlim N_2$ (and it is a standard fact that $\varinjlim (N_1\oplus N_2)\cong \varinjlim N_1 \oplus \varinjlim N_2$ \cite[Theorem V.5.1]{mac2013categories}). Therefore, we have that \[c=\rk(M)(I)=\rank(M|_I)=\sum_{j\in L}\sum_{k\in K_j}\rank\left(M_{jk}\right).\]
Since each $M_{jk}$ is indecomposable, by Theorem \ref{thm:rk}, \[\rank(M_{jk})=\begin{cases}1,&\mbox{if $M_{jk}\cong \I_{I}$} \\0,&\mbox{otherwise.}
\end{cases}
\]
Therefore, there exist exactly $c$ distinct pairs $(j,k)$ such that $M_{jk}\cong \I_I$. Hence, for each of such $(j,k)$, we can find a section $\bv:=\bv^{jk}$ of $M_{jk}$ that is fully supported over $I$. Since $M_{jk}$ is a summand of $M_j|_I$, $\bv$ is also a section of $M_j|_I$, as desired. 
\end{proof}

\subsection{Barcode ensemble from {\sc Interval}}

Fix a $P$-module $M$ as an input to the algorithm {\sc Interval}. Recall that an output of {\sc Interval} with the input $M$ is a collection $\{(I_i,\mu_{I_i})\}_i$.
This collection may change with different choices available during the exploration for computing maximal intervals. The collection $\bE(M)$  of \emph{all} possible outputs of {\sc Interval} will be called the \textbf{barcode ensemble} of $M$. Proposition \ref{prop:two algorithm coincides} implies:

\begin{corollary} If $\bE(M)$ contains more than one collection, then $M$ is not interval decomposable. 
\end{corollary}

Suppose that a collection $\Ccal:=\{(I_i,\mu_{I_i})\}_i$ belongs to $\bE(M)$. We consider the interval decomposable module  $\mathbb{I}_\Ccal:=\bigoplus_i \mathbb{I}_{I_i}^{\mu_i}$ corresponding to $\Ccal$. Fix any interval $J\subset P$. We define \[\bE(M)(J):=\max_{\Ccal \in \bE(M)} \rk(\bbI_{\Ccal})(J).\]
We remark that $\rk(\bbI_{\Ccal})(J)$ is equal to $\sum_{i} \mu_i$ where the sum is taken over all $i$ such that $I_i\supseteq J$.

A single output of the algorithm {\sc Interval} may fail to capture the rank invariant or generalized rank invariant of $M$; in Example \ref{ex:intervals2}, the rank invariant of the module $N$  is different from the rank invariant of $\bigoplus_{i=1}^3\bbI^{J_i}$. Nevertheless, the barcode ensemble $\bE(M)$ recovers the generalized rank invariant of the input module $M$:

\begin{proposition}\label{prop:rk is recovered} For every interval $J\subseteq P$, we have:
\[\rk(M)(J)=\bE(M)(J).\]
\end{proposition}
By the design of the algorithm {\sc Interval}, the above equality is clearly  true when $\rk(M)(J)=0$.
\begin{proof}
($\geq$) Let $m:=\rk(M)(J)$ and pick any $\Ccal=\{(I_i,\mu_i))\}_i$ in $\bE(M)$. By the design of the algorithm {\sc Interval}, it is not possible for the sum $\displaystyle \sum_{\substack{i \\ I_i\supset J}} \mu_i$ ($=\rk(\bbI_{\C})$) to be greater than $m$. Since $\Ccal$ was arbitrarily chosen in $\bE(M)$, we have $\rk(M)(J)\geq \bE(M)(J)$.

($\leq$) Assume that $m:=\bE(M)(J)$. This implies that there exists $\Ccal=\{(I_i,\mu_i)\}_i$ in $\bE(M)$ such that $\rk(\bbI_\C)(J)=m$. We prove the desired inequality by contradiction. Suppose that $\rk(M)(J)>m$. Then, by starting the algorithm {\sc Interval} at any point $p\in I$ (in Step 2) and inductively adding points to $\{p\}$ from $I\setminus \{p\}$ whenever possible (in Step 2.2), the algorithm will output intervals $I'_j$ containing $J$ with multiplicity $\mu'_j$ such that the sum $\sum_{j}\mu'_j$ is larger than $m$. Let $\Ccal'$ be the output obtained in this strategy. Then, we have that $\rk(\bbI_{\Ccal'})(J)>m$ which contradicts the definition of $\bE(M)(J)$.
\end{proof}

\section{Discussion}\label{sec:discussion}

Some open questions that follow are:(i) Can we generalize Theorem \ref{thm:rank is the multiplicity of full intervals} to $d$-parameter persistent homology for $d>2$? (ii) Can the complexity of the algorithms be improved? (iii) In particular, can we improve the interval decomposability testing by improving {\sc TestInterval} which currently uses an algorithm of Asashiba et al.?

\appendix

\section{Limits and Colimits}\label{sec:limits and colimits}

We recall the notions of limit and colimit from category theory \cite{mac2013categories}. Recall that a poset $\Pb$ can be viewed as a category whose objects are the elements of $\Pb$ and morphisms are the comparable pairs $p\leq q$ in $\Pb$. Although limits and colimits are defined for functors indexed by small categories we restrict our attention to poset-indexed functors. Let $\C$ be any category.  
\begin{definition}[Cone]\label{def:cone} Let $F:\Pb\rightarrow \C$ be a functor. A \emph{cone} over $F$ is a pair $\left(L,(\pi_p)_{p\in \Pb}\right)$ consisting of an object $L$ in $\C$ and a collection $(\pi_p)_{p\in \Pb}$ of morphisms $\pi_p:L \rightarrow F(p)$ that commute with the arrows in the diagram of $F$, i.e. if $p\leq q$ in $\Pb$, then $\pi_q= F(p\leq q)\circ \pi_p$ in $\C$, i.e. the diagram below commutes.
\end{definition}
\begin{equation}\label{eq:cone commutativity}
    \begin{tikzcd}F(p)\arrow{rr}{F(p\leq q)}&&F(q)\\
& L \arrow{lu}{\pi_p} \arrow{ru}[swap]{\pi_q}\end{tikzcd}
\end{equation}

In Definition \ref{def:cone}, the cone $\left(L,(\pi_p)_{p\in \Pb}\right)$ over $F$ is sometimes denoted simply by $L$, suppressing the collection $(\pi_p)_{p\in \Pb}$ of morphisms if no confusion can arise. A limit of $F:\Pb\rightarrow \C$ is a terminal object in the collection of all cones over $F$:

\begin{definition}[Limit]\label{def:limit} Let $F:\Pb\rightarrow \C$ be a functor. A \emph{limit} of $F$ is a cone over $F$, denoted by $\left(\varprojlim F,\ (\pi_p)_{p\in \Pb} \right)$ or simply $\varprojlim F$, with the following \emph{terminal} property: If there is another cone $\left(L',(\pi'_p)_{p\in \Pb} \right)$ of $F$, then there is a \emph{unique} morphism $u:L'\rightarrow \varprojlim F$ such that $\pi_p'=\pi_p\circ u$ for all $p\in \Pb$.
\end{definition}

It is possible that a functor does not have a limit at all. However, if a functor does have a limit then the terminal property of the limit guarantees its uniqueness up to isomorphism. For this reason, we  sometimes refer to a limit as \emph{the} limit of a functor.

Cocones and colimits are defined in a dual manner:
\begin{definition}[Cocone]\label{def:cocone} Let $F:\Pb\rightarrow \C$ be a functor. A \emph{cocone} over $F$ is a pair $\left(C,(i_p)_{p\in \Pb}\right)$ consisting of an object $C$ in $\C$ and a collection $(i_p)_{p\in \Pb}$ of morphisms $i_p:F(p)\rightarrow C$ that commute with the arrows in the diagram of $F$, i.e. if $p\leq q$ in $\Pb$, then $i_p= i_q\circ F(p\leq q)$ in $\C$, i.e. the diagram below commutes.
\end{definition}
\begin{equation}\label{eq:cocone commutativity}
    \begin{tikzcd}
&C\\
F(p)\arrow{ru}{i_p}\arrow{rr}[swap]{F(p\leq q)}&&F(q)\arrow{lu}[swap]{i_q}
\end{tikzcd}
\end{equation}

In Definition \ref{def:cocone}, a cocone $\left(C,(i_p)_{p\in \Pb}\right)$ over $F$ is sometimes denoted simply by $C$, suppressing the collection $(i_p)_{p\in \Pb}$ of morphisms. A colimit of $F:\Pb\rightarrow \C$ is an initial object in the collection of cocones over $F$:

\begin{definition}[Colimit]\label{def:colimit}Let $F:\Pb\rightarrow \C$ be a functor. A \emph{colimit} of $F$ is a cocone, denoted by $\left(\varinjlim F,\ (i_p)_{p\in \Pb}\right)$ or simply $\varinjlim F$, with the following \emph{initial} property: If there is another cocone $\left(C', (i'_p)_{p\in \Pb}\right)$ of $F$, then there is a \emph{unique} morphism $u:\varinjlim F\rightarrow C'$  such that $i'_p=u\circ i_p$ for all $p\in \Pb$.
\end{definition}

It is possible that a functor does not have a colimit at all. However, if a functor does have a colimit then the initial property of the colimit guarantees its uniqueness up to isomorphism. For this reason, we sometimes refer to a colimit as \emph{the} colimit of a functor.  
It is well-known that if $P$ is finite, then any functor $F:\Pb\rightarrow \vect$ admits both limit and colimit in $\vect$. Assume that $\Pb$ is also connected. Then, by the commutativity in (\ref{eq:cone commutativity}) and (\ref{eq:cocone commutativity}), once a cone and a cocone of $F$ are specified, there exists the canonical map from the cone to the cocone, leading to Definition \ref{def:rank}.

\begin{remark}\label{rem:canonical projection}Let $\Pb$ be a finite and connected poset. Let $\Q$ be a finite and connected subposet of $\Pb$. Let us fix any $F:\Pb \rightarrow \vect$.
\begin{enumerate}[label=(\roman*)]
    \item  For any cone $\left(L', (\pi_p')_{p\in \Pb}\right)$ over $F$, its restriction $\left(L', (\pi_p')_{p\in \Q}\right)$ is a cone over the restriction $F|_\Q:\Q\rightarrow \vect$. Therefore, by the terminal property of the limit  $\left(\varprojlim F|_{\Q}, (\pi_q)_{q\in \Q}\right)$, there exists the unique morphism $u:L' \rightarrow \varprojlim F|_{\Q}$ such that  $\pi_q'=\pi_q\circ u$ for all $q\in \Q$.\label{item:canonical projection 1}
    \item  For any cocone $\left(C', (i_p')_{p\in \Pb}\right)$ over $F$, its restriction $\left(C', (i_p')_{p\in \Q}\right)$ is a cocone over the restriction $F|_\Q:\Q\rightarrow \vect$. Therefore, by the initial property of $\varinjlim F|_{\Q}$, there exists the unique morphism $u:\varinjlim F|_{\Q} \rightarrow C'$ such that $i'_q=u\circ i_q$ for all $q\in \Q$.\label{item:canonical projection 2}
     \item By the previous two items, there exist linear maps $\pi:\varprojlim F\rightarrow \varprojlim F|_{\Q}$ and $\iota:\varinjlim F|_{\Q}\rightarrow \varinjlim F$ such that
     $\psi_F=\iota \circ \psi_{F|_{\Q}} \circ \pi.$\label{item:canonical projection 3}
\end{enumerate}
Therefore, $\rank(F)=\rank(\psi_{F})\leq \rank (\psi_{F|_{\Q}})=\rank(F|_{\Q})$.
\end{remark}

\bibliographystyle{plain}
\bibliography{biblio.bib}

\begin{thebibliography}{10}

\bibitem{asashiba2018interval}
Hideto Asashiba, Micka{\"e}l Buchet, Emerson~G Escolar, Ken Nakashima, and
  Michio Yoshiwaki.
\newblock On interval decomposability of 2d persistence modules.
\newblock {\em arXiv preprint arXiv:1812.05261v2}, 2018.

\bibitem{asashiba2019approximation}
Hideto Asashiba, Emerson~G Escolar, Ken Nakashima, and Michio Yoshiwaki.
\newblock On approximation of $2d$-persistence modules by
  interval-decomposables.
\newblock {\em arXiv preprint arXiv:1911.01637}, 2019.

\bibitem{azumaya1950corrections}
Gor{\^o} Azumaya.
\newblock Corrections and supplementaries to my paper concerning
  {K}rull-{R}emak-{S}chmidt’s theorem.
\newblock {\em Nagoya Mathematical Journal}, 1:117--124, 1950.

\bibitem{bauer2020cotorsion}
Ulrich Bauer, Magnus~B Botnan, Steffen Oppermann, and Johan Steen.
\newblock Cotorsion torsion triples and the representation theory of filtered
  hierarchical clustering.
\newblock {\em Advances in Mathematics}, 369:107171, 2020.

\bibitem{betthauser2021graded}
Leo Betthauser, Peter Bubenik, and Parker~B Edwards.
\newblock Graded persistence diagrams and persistence landscapes.
\newblock {\em Discrete \& Computational Geometry}, pages 1--28, 2021.

\bibitem{biasotti2008multidimensional}
Silvia Biasotti, Andrea Cerri, Patrizio Frosini, Daniela Giorgi, and Claudia
  Landi.
\newblock Multidimensional size functions for shape comparison.
\newblock {\em Journal of Mathematical Imaging and Vision}, 32(2):161--179,
  2008.

\bibitem{botnan2020rectangle}
Magnus Botnan, Vadim Lebovici, and Steve Oudot.
\newblock {On Rectangle-Decomposable 2-Parameter Persistence Modules}.
\newblock In {\em 36th International Symposium on Computational Geometry (SoCG
  2020)}, volume 164, pages 22:1--22:16, 2020.

\bibitem{botnan2018algebraic}
Magnus Botnan and Michael Lesnick.
\newblock Algebraic stability of zigzag persistence modules.
\newblock {\em Algebraic \& geometric topology}, 18(6):3133--3204, 2018.

\bibitem{botnan2021signed}
Magnus Botnan, Steffen Oppermann, and Steve Oudot.
\newblock Signed barcodes for multi-parameter persistence via rank
  decompositions and rank-exact resolutions.
\newblock {\em arXiv preprint arXiv:2107.06800}, 2021.

\bibitem{bubenik2020virtual}
Peter Bubenik and Alex Elchesen.
\newblock Virtual persistence diagrams, signed measures, and wasserstein
  distance.
\newblock {\em arXiv preprint arXiv:2012.10514}, 2020.

\bibitem{cai2020elder}
Chen Cai, Woojin Kim, Facundo M{\'e}moli, and Yusu Wang.
\newblock Elder-rule-staircodes for augmented metric spaces.
\newblock {\em SIAM Journal on Applied Algebra and Geometry}, 5(3):417--454,
  2021.

\bibitem{carlsson2010zigzag}
Gunnar Carlsson and Vin de~Silva.
\newblock Zigzag persistence.
\newblock {\em Foundations of computational mathematics}, 10(4):367--405, 2010.

\bibitem{carlsson2010multiparameter}
Gunnar Carlsson and Facundo M{\'e}moli.
\newblock Multiparameter hierarchical clustering methods.
\newblock In {\em Classification as a Tool for Research}, pages 63--70.
  Springer, 2010.

\bibitem{carlsson2009theory}
Gunnar Carlsson and Afra Zomorodian.
\newblock The theory of multidimensional persistence.
\newblock {\em Discrete \& Computational Geometry}, 42(1):71--93, 2009.

\bibitem{chambers2018persistent}
Erin Chambers and David Letscher.
\newblock Persistent homology over directed acyclic graphs.
\newblock In {\em Research in Computational Topology}, pages 11--32. Springer,
  2018.

\bibitem{cochoy2020decomposition}
J{\'e}r{\'e}my Cochoy and Steve Oudot.
\newblock Decomposition of exact pfd persistence bimodules.
\newblock {\em Discrete \& Computational Geometry}, 63(2):255--293, 2020.

\bibitem{DH21}
Tamal~K. Dey and Tao Hou.
\newblock Updating zigzag persistence and maintaining representatives over
  changing filtrations.
\newblock {\em CoRR}, abs/2112.02352, 2021.

\bibitem{DW22}
Tamal~K. Dey and Yusu Wang.
\newblock {\em Computational Topology for Data Analysis}.
\newblock Cambridge University Press, 2022.
\newblock
  \url{https://www.cs.purdue.edu/homes/tamaldey/book/CTDAbook/CTDAbook.pdf}.

\bibitem{dey2019generalized}
Tamal~K Dey and Cheng Xin.
\newblock Generalized persistence algorithm for decomposing multiparameter
  persistence modules.
\newblock {\em Journal of Applied and Computational Topology}, pages 1--52,
  2022.

\bibitem{EH2010}
Herbert Edelsbrunner and John Harer.
\newblock {\em Computational Topology: An Introduction}.
\newblock American Mathematical Society, Jan 2010.

\bibitem{escolar2016persistence}
Emerson~G Escolar and Yasuaki Hiraoka.
\newblock Persistence modules on commutative ladders of finite type.
\newblock {\em Discrete \& Computational Geometry}, 55(1):100--157, 2016.

\bibitem{gabrielsthm}
Pierre Gabriel.
\newblock Unzerlegbare darstellungen i.
\newblock {\em Manuscripta Mathematica}, pages 71--103, 1972.

\bibitem{kerber2020multi}
Michael Kerber.
\newblock Multi-parameter persistent homology is practical.
\newblock In {\em NeurIPS 2020 Workshop on Topological Data Analysis and
  Beyond}, 2020.

\bibitem{kim2018generalized_v1}
Woojin Kim and Facundo M{\'e}moli.
\newblock Rank invariant for zigzag modules.
\newblock {\em arXiv preprint arXiv:1810.11517v1}, 2018.

\bibitem{kim2018generalized}
Woojin Kim and Facundo M{\'e}moli.
\newblock Generalized persistence diagrams for persistence modules over posets.
\newblock {\em Journal of Applied and Computational Topology}, 5(4):533--581,
  2021.

\bibitem{kim2021spatiotemporal}
Woojin Kim and Facundo M{\'e}moli.
\newblock Spatiotemporal persistent homology for dynamic metric spaces.
\newblock {\em Discrete \& Computational Geometry}, 66(3):831--875, 2021.

\bibitem{kim2021bigraded}
Woojin Kim and Samantha Moore.
\newblock The generalized persistence diagram encodes the bigraded {B}etti
  numbers.
\newblock {\em arXiv preprint arXiv:2111.02551}, 2021.

\bibitem{lesnick2012multidimensional}
Michael Lesnick.
\newblock {\em Multidimensional interleavings and applications to topological
  inference}.
\newblock Stanford University, 2012.

\bibitem{lesnick2015theory}
Michael Lesnick.
\newblock The theory of the interleaving distance on multidimensional
  persistence modules.
\newblock {\em Foundations of Computational Mathematics}, 15(3):613--650, 2015.

\bibitem{lesnick2015interactive}
Michael Lesnick and Matthew Wright.
\newblock Interactive visualization of 2-d persistence modules.
\newblock {\em arXiv preprint arXiv:1512.00180}, 2015.

\bibitem{mac2013categories}
Saunders Mac~Lane.
\newblock {\em Categories for the working mathematician}, volume~5.
\newblock Springer Science \& Business Media, 2013.

\bibitem{mccleary2020edit}
Alexander McCleary and Amit Patel.
\newblock Edit distance and persistence diagrams over lattices.
\newblock {\em arXiv preprint arXiv:2010.07337}, 2020.

\bibitem{miller2019modules}
Ezra Miller.
\newblock Modules over posets: commutative and homological algebra.
\newblock {\em arXiv preprint arXiv:1908.09750}, 2019.

\bibitem{milosavljevic2011zigzag}
Nikola Milosavljevi{\'c}, Dmitriy Morozov, and Primoz Skraba.
\newblock Zigzag persistent homology in matrix multiplication time.
\newblock In {\em Proceedings of the twenty-seventh Annual Symposium on
  Computational Geometry}, pages 216--225, 2011.

\bibitem{munkres2018elements}
James~R Munkres.
\newblock {\em Elements of algebraic topology}.
\newblock CRC press, 2018.

\bibitem{patel2018generalized}
Amit Patel.
\newblock Generalized persistence diagrams.
\newblock {\em Journal of Applied and Computational Topology}, 1(3):397--419,
  2018.

\bibitem{rota1964foundations}
Gian-Carlo Rota.
\newblock On the foundations of combinatorial theory i. theory of {M}{\"o}bius
  functions.
\newblock {\em Zeitschrift f{\"u}r Wahrscheinlichkeitstheorie und verwandte
  Gebiete}, 2(4):340--368, 1964.

\bibitem{stanley2011enumerative}
Richard~P Stanley.
\newblock Enumerative combinatorics volume 1 second edition.
\newblock {\em Cambridge studies in advanced mathematics}, 2011.

\end{thebibliography}

\end{document}